\documentclass{article}
\usepackage[margin=3cm]{geometry}
\usepackage{graphicx} 
\usepackage{amsmath}
\usepackage{amssymb}
\usepackage[utf8]{inputenc}
\usepackage{amsfonts}
\usepackage{amsthm}
\usepackage{hyperref}
\usepackage{mathabx}
\usepackage{dsfont}
\usepackage{xcolor}
\usepackage{tikz}
\usetikzlibrary{shapes, decorations}
\usepackage{qtree}
\usepackage{multicol}
\usepackage{mathabx}
\usepackage{faktor}
\usepackage{mathtools}

\theoremstyle{plain}
\newtheorem{theorem}{Theorem} 
\newtheorem{lemma}{Lemma} 
\newtheorem{prop}{Propositon}
\newtheorem{cor}{Corollary}
 
\newtheorem{remark}{Remark}

\theoremstyle{definition}
\newtheorem{defn}{Definition} 
\theoremstyle{remark}

\newcommand{\sing}{\text{sing}}

\newcommand{\defref}[1]{\hyperref[#1]{Definition \ref*{#1}}}
\newcommand{\Defref}[1]{\hyperref[#1]{Definition \ref*{#1}}}
\newcommand{\lemref}[1]{\hyperref[#1]{Lemma \ref*{#1}}}
\newcommand{\Lemref}[1]{\hyperref[#1]{Lemma \ref*{#1}}}
\newcommand{\thmref}[1]{\hyperref[#1]{Theorem \ref*{#1}}}
\newcommand{\Thmref}[1]{\hyperref[#1]{Theorem \ref*{#1}}}
\newcommand{\corref}[1]{\hyperref[#1]{Corollary \ref*{#1}}}
\newcommand{\Corref}[1]{\hyperref[#1]{Corollary \ref*{#1}}}
\newcommand{\propref}[1]{\hyperref[#1]{Propositon \ref*{#1}}}
\newcommand{\figref}[1]{\hyperref[#1]{Figure \ref*{#1}}}
\newcommand{\rmkref}[1]{\hyperref[#1]{Remark \ref*{#1}}}
\newcommand{\secref}[1]{\hyperref[#1]{Section \ref*{#1}}}

\tikzset{
          solid circle/.style={circle,draw,inner sep=1.2,fill=black},
          solid green circle/.style={circle,draw,inner sep=1.2,fill=green, draw=green},
          solid red circle/.style={circle,draw,inner sep=1.2,fill=red, draw=red},
          solid blue circle/.style={circle,draw,inner sep=1.2,fill=blue, draw=blue},
          solid yellow circle/.style={circle,draw,inner sep=1.2,fill=yellow, draw=yellow},
          solid cyan circle/.style={circle,draw,inner sep=1.2,fill=cyan, draw=cyan},
          big solid circle/.style={circle,draw,inner sep=1.2,fill=black, minimum size=0.3cm}, 
          big solid green circle/.style={circle,draw,inner sep=1.2,fill=green, draw=green, minimum size=0.3cm},
          big solid red circle/.style={circle,draw,inner sep=1.2,fill=red, draw=red, minimum size=0.3cm},
          big solid blue circle/.style={circle,draw,inner sep=1.2,fill=blue, draw=blue, minimum size=0.3cm},
          big solid yellow circle/.style={circle,draw,inner sep=1.2,fill=yellow, draw=yellow, minimum size=0.3cm},
          big solid cyan circle/.style={circle,draw,inner sep=1.2,fill=cyan, draw=cyan, minimum size=0.3cm},
          hollow circle/.style={circle,draw,inner sep=1.2},
          green hollow circle/.style={circle,draw,inner sep=1.2, draw=green},
          red hollow circle/.style={circle,draw,inner sep=1.2, draw=red},
          blue hollow circle/.style={circle,draw,inner sep=1.2, draw=blue},
          yellow hollow circle/.style={circle,draw,inner sep=1.2, draw=yellow},
          cyan hollow circle/.style={circle,draw,inner sep=1.2, draw=cyan},
          big hollow circle/.style={circle,draw,inner sep=1.2, minimum size=0.3cm},
          big hollow green circle/.style={circle,draw,inner sep=1.2, draw=green, minimum size=0.3cm},
          big hollow red circle/.style={circle,draw,inner sep=1.2, draw=red, minimum size=0.3cm},
          big hollow blue circle/.style={circle,draw,inner sep=1.2, draw=blue, minimum size=0.3cm},
          big hollow yellow circle/.style={circle,draw,inner sep=1.2, draw=yellow, minimum size=0.3cm},
          big hollow cyan circle/.style={circle,draw,inner sep=1.2, draw=cyan, minimum size=0.3cm},
          solid triangle/.style={regular polygon, regular polygon sides=3,draw,inner sep=1.2,fill=black},
          green solid triangle/.style={regular polygon, regular polygon sides=3,draw,inner sep=1.2,fill=green, draw=green},
          solid red triangle/.style={regular polygon, regular polygon sides=3,draw,inner sep=1.2,fill=red, draw=red},
          solid blue triangle/.style={regular polygon, regular polygon sides=3,draw,inner sep=1.2,fill=blue, draw=blue},
          solid yellow triangle/.style={regular polygon, regular polygon sides=3,draw,inner sep=1.2,fill=yellow, draw=yellow},
          solid cyan triangle/.style={regular polygon, regular polygon sides=3,draw,inner sep=1.2,fill=cyan, draw=cyan},
          big solid triangle/.style={regular polygon, regular polygon sides=3,draw,inner sep=1.2,fill=black, minimum size=0.3cm},
          big solid green triangle/.style={regular polygon, regular polygon sides=3,draw,inner sep=1.2,fill=green, draw=green, minimum size=0.3cm},
          big solid red triangle/.style={regular polygon, regular polygon sides=3,draw,inner sep=1.2,fill=red, draw=red, minimum size=0.3cm},
          big solid blue triangle/.style={regular polygon, regular polygon sides=3,draw,inner sep=1.2,fill=blue, draw=blue, minimum size=0.3cm},
          big solid yellow triangle/.style={regular polygon, regular polygon sides=3,draw,inner sep=1.2,fill=yellow, draw=yellow, minimum size=0.3cm},
          big solid cyan triangle/.style={regular polygon, regular polygon sides=3,draw,inner sep=1.2,fill=cyan, draw=cyan, minimum size=0.3cm},
          hollow triangle/.style={regular polygon, regular polygon sides=3,draw,inner sep=1.2},
          green hollow triangle/.style={regular polygon, regular polygon sides=3,draw,inner sep=1.2, draw=green},
          red hollow triangle/.style={regular polygon, regular polygon sides=3,draw,inner sep=1.2, draw=red},
          blue hollow triangle/.style={regular polygon, regular polygon sides=3,draw,inner sep=1.2, draw=blue},
          yellow hollow triangle/.style={regular polygon, regular polygon sides=3,draw,inner sep=1.2, draw=yellow},
          cyan hollow triangle/.style={regular polygon, regular polygon sides=3,draw,inner sep=1.2, draw=cyan},
          big hollow triangle/.style={regular polygon, regular polygon sides=3,draw,inner sep=1.2, minimum size=0.3cm},
          big hollow green triangle/.style={regular polygon, regular polygon sides=3,draw,inner sep=1.2, draw=green, minimum size=0.3cm},
          big hollow red triangle/.style={regular polygon, regular polygon sides=3,draw,inner sep=1.2, draw=red, minimum size=0.3cm},
          big hollow blue triangle/.style={regular polygon, regular polygon sides=3,draw,inner sep=1.2, draw=blue, minimum size=0.3cm},
          big hollow yellow triangle/.style={regular polygon, regular polygon sides=3,draw,inner sep=1.2, draw=yellow, minimum size=0.3cm},
          big hollow cyan triangle/.style={regular polygon, regular polygon sides=3,draw,inner sep=1.2, draw=cyan, minimum size=0.3cm},
          solid square/.style={regular polygon, regular polygon sides=4,draw,inner sep=1.2,fill=black},
          solid green square/.style={regular polygon, regular polygon sides=4,draw,inner sep=1.2,fill=green, draw=green},
          solid red square/.style={regular polygon, regular polygon sides=4,draw,inner sep=1.2,fill=red, draw=red},
          solid blue square/.style={regular polygon, regular polygon sides=4,draw,inner sep=1.2,fill=blue, draw=blue},
          solid yellow square/.style={regular polygon, regular polygon sides=4,draw,inner sep=1.2,fill=yellow, draw=yellow},
          solid cyan square/.style={regular polygon, regular polygon sides=4,draw,inner sep=1.2,fill=cyan, draw=cyan},
          big solid square/.style={regular polygon, regular polygon sides=4,draw,inner sep=1.2,fill=black, minimum size=0.3cm},
          big solid green square/.style={regular polygon, regular polygon sides=4,draw,inner sep=1.2,fill=green, draw=green, minimum size=0.3cm},
          big solid red square/.style={regular polygon, regular polygon sides=4,draw,inner sep=1.2,fill=red, draw=red, minimum size=0.3cm},
          big solid blue square/.style={regular polygon, regular polygon sides=4,draw,inner sep=1.2,fill=blue, draw=blue, minimum size=0.3cm},
          big solid yellow square/.style={regular polygon, regular polygon sides=4,draw,inner sep=1.2,fill=yellow, draw=yellow, minimum size=0.3cm},
          big solid cyan square/.style={regular polygon, regular polygon sides=4,draw,inner sep=1.2,fill=cyan, draw=cyan, minimum size=0.3cm},
          hollow square/.style={regular polygon, regular polygon sides=4,draw,inner sep=1.2},
          green hollow square/.style={regular polygon, regular polygon sides=4,draw,inner sep=1.2, draw=green},
          red hollow square/.style={regular polygon, regular polygon sides=4,draw,inner sep=1.2, draw=red},
          blue hollow square/.style={regular polygon, regular polygon sides=4,draw,inner sep=1.2, draw=blue},
          yellow hollow square/.style={regular polygon, regular polygon sides=4,draw,inner sep=1.2, draw=yellow},
          cyan hollow square/.style={regular polygon, regular polygon sides=4,draw,inner sep=1.2, draw=cyan},
          big hollow square/.style={regular polygon, regular polygon sides=4,draw,inner sep=1.2, minimum size=0.3cm},
          big hollow green square/.style={regular polygon, regular polygon sides=4,draw,inner sep=1.2, draw=green, minimum size=0.3cm},
          big hollow red square/.style={regular polygon, regular polygon sides=4,draw,inner sep=1.2, draw=red, minimum size=0.3cm},
          big hollow blue square/.style={regular polygon, regular polygon sides=4,draw,inner sep=1.2, draw=blue, minimum size=0.3cm},
          big hollow yellow square/.style={regular polygon, regular polygon sides=4,draw,inner sep=1.2, draw=yellow, minimum size=0.3cm},
          big hollow cyan square/.style={regular polygon, regular polygon sides=4,draw,inner sep=1.2, draw=cyan, minimum size=0.3cm},
          left label/.style={above left,midway},
          right label/.style={above right,midway}
        }

\title{Gluing diagrams part 1:\\ A constructive solution for the Higman-Thompson group isomorphism problem}
\author{Roman Gorazd\\ \url{roman.gorazd@gmail.com}}
\date{May 2026         }

\begin{document}

\maketitle
\begin{abstract}
	This paper introduces gluing diagrams a combinatorial tool to construct homomorphisms between the shift pseudogroups of directed graphs and thus also their full groups of shifts. We will establish which of these diagrams produce isomorphisms. As an application, using the interpretation of Higman-Thompson groups as full groups of shifts of specific graphs, we will describe a procedure that constructs gluing diagrams that explicitly describe the isomorphisms between Higman-Thompson groups, conjectured by Higman\cite{higman74} and whose existence was proven by Pardo\cite{Pardo2011}.
\end{abstract}
\section{Introduction}\label{intro-sec}
Higman-Thompson groups \(V_{n,r}\) were introduced by Graham Higman in 1974\cite{higman74} as a generalization of Thompson's \(V\) group\cite{cannon1996introductory}. In the same paper he proved that if \(V_{n,r}\cong V_{m,s}\) we have \(n=m\) and \(\gcd(r,n-1)=\gcd(s,n-1)\) and conjectured the converse implication. In 2010 Enrique Pardo\cite{Pardo2011} proved the converse implication by using results from the study of Leavitt path algebras\cite{abrams2015leavitt}, without giving an explicit way of construct the desired isomorphisms. 

This paper we will give a combinatorial way of constructing these isomorphisms, via a very elementary approach. To do this we will view the Higman-Thompson groups as consisting of almost automorphisms on quasi regular trees as in \cite{scott84}, which allows us to adopt a much more general approach which will enable us to describe homomorphisms between more general objects. These objects will be shift pseudogroups (also defined in \cite{ARKLINT_EILERS_RUIZ_2018}) acting on the path monoid of a directed rooted graph (which can also be seen as the monoid of clopen sets of the boundary of the path spaces of a graph), which will be defined in \secref{def-sec}. In \secref{gluing-diag-sec} we will introduce the tool of gluing diagrams, which allows us to describe homomorphisms of the path monoid that are equivariant with respect to the action of shift pseudogroups and thus induce a homomorphism between them as well. 

This homomorphism will restrict to a homomorphism between the full group of the one-sided shift groupoid\cite{Matui2015}, which in the special case of graphs with one source and one other vertex are the Higman-Thompson groups. Through \lemref{non-inj-lem}, \propref{surj-lem} and \thmref{shift-surj-thm} we will specify the properties of the gluing diagram that are sufficient (and also necessary in many cases) for the homomorphism of path monoids described by it to be an isomorphism and for the homomorphism induced on the shift pseudogroup to be an isomorphism. We will also introduce a way of modify these diagrams without changing the described function.

Finally, in \secref{H-T-sec} focusing on the special case of Higman-Thompson groups we will introduce moves on the gluing diagrams between graphs with a source vertex and one other vertex  that keep the induced function an isomorphism. These moves will allow us to follow the Euclidean Algorithm to construct gluing diagrams that induce isomorphisms \(V_{n,r}\cong V_{n,s}\) whenever \(\gcd(r,n-1)=\gcd(s,n-1)\).

In the sequels to this paper we will explore which homomorphisms of the shift pseudogroup are described by gluing diagrams and how we can realize the homomorphisms described by gluing diagrams by graph moves.

\section{Definitions}\label{def-sec}
This section will introduce the key concepts of the Paper that will allow us to construct gluing diagrams and the homomorphisms they describe. We will first define graphs as directed graphs with multiple edges and loops, together with the space of their paths, endowed with the prefix order. The graph monoid will be generated by the vertices and the path monoid by the (rooted) paths, with their defining relations being analogous. They will be furthermore connected by the terminus homomorphism from the path monoid to the graph monoid, that extends the notion of a terminus of a path. We will define shifts to be certain partial isomorphism on the path monoid, that form a pseudogroup. After establishing some fundamental facts about the path monoid and shifts, we will be ready to construct gluing diagrams and show how they describe homomorphisms of the path monoid and shift groupoids, in the next chapter. 

Formally we define graphs as follows.
\begin{defn}\label{graph-defn}
    A graph \(G\) is a tuple \((VG,EG,o_G,t_G)\), where \(VG,EG\) are sets and \(o_G,t_G:EG\to VG\) are functions. We will call  \(VG\) vertices, \(EG\) edges, \(o_G(e)\) the origin and \(t_G(e)\) the terminus of an edge \(e\). We will call vertices \(v\in VG\) regular if \(0<|o^{-1}(v)|<\infty\) and singular otherwise. Denote the set of regular vertices by \(VG_{\text{reg}}\) and the set of singular vertices by \(VG_{\text{sing}}\). We will call vertex \(v\in VG\) a \textbf{sink} if \(o^{-1}(v)=\emptyset\), and a \textbf{source} if \(t^{-1}(v)=\emptyset\).
\end{defn}
If the graph we are working in is clear we will drop the subscript \(G\) for notational convenience. We will be assuming throughout the paper that any graph is \(o\)-finite i.e. \(\forall v \in VG,\ |o^{-1}(v)|<\infty\).  
\begin{defn}\label{graph-monoid-defn}
    For any graph \(G\) we define its graph monoid \(\mathcal{M}_G\) to be the commutative monoid generated by \(VG\) with the relations 
        \[\forall v\in VG_{\text{reg}},\ v= \sum_{e\in o^{-1}(v)}t(e)\] 
\end{defn}

Any graph gives rise to a set of paths
    \[\mathcal{P}(G)=\{(e_i)^n_{i=1}\in (EG)^{*}\mid \forall 1\leq i < n\ t(e_i)=o(e_{i+1})\}\cup\{\varepsilon_v\mid v\in VG\},\]
where the elements \(\varepsilon_v\) denote the empty path based on the vertex \(v\). Each non-empty path \(p=(e_i)^n_{i=1}\in \mathcal{P}(G)\) has an origin \(O(p):=o(e_1)\), terminus \(T(p):=t(e_n)\) and length \(|p|:=n\), and additionally we have \(T(\varepsilon_v)=O(\varepsilon_v):=v\), \(|\varepsilon_v|:=0\) for any vertex \(v\in VG\). We can concatenate two non-empty paths \(p=(e_i)^n_{i=1}, q=(d_j)^m_{j=1}\in\mathcal{P}(G)\) if \(T(p)=O(q)\) by setting
	\[pq:=e_1\dots e_nd_1\dots d_m\in\mathcal{P}(G),\]
for empty paths we will set
	\[p\varepsilon_{T(p)}=\varepsilon_{O(p)}p:=p\]
for any \(p\in\mathcal{P}(G)\).
	
This allows us to define the prefix order \(\preceq\) on the path space by
	\[p\preceq q \iff \exists r\in\mathcal{P}(G),\ q=pr.\]
This order makes \(\mathcal{P}(G)\) a meet semilattice. We will call two paths \(p,q\in \mathcal{P}(G)\) independent if \(p\not\preceq q\) and \(q\not\preceq p\), we then write \(p\perp q\).  We will call a path singular if \(T(p)\in VG_{\text{sing}}\) and regular otherwise. Additionally, we can view the path space as a graph by defining \(E\mathcal{P}(G)= \{(p,pe)\mid p\in \mathcal{P}(G),e\in o_G^{-1}(T(p))\}\) with \(o_{\mathcal{P}(G)}((p,pe))=p\) and \(t_{\mathcal{P}(G)}((p,pe))=pe\).

Using this we define for each vertex \(v\in VG\) the subsemilattice 
    \[\mathcal{P}(G,v):=O^{-1}(\{v\})= \{p\in \mathcal{P}(G)\mid O(p)=v\},\]
i.e. the collection of all paths that originate in \(v\). \(\mathcal{P}(G,v)\) can be of course seen as a subgraph of \(\mathcal{P}(G)\). We will mostly consider graphs with roots i.e. vertices \(R\in VG\) s.t.
    \[\forall v\in VG, \exists p\in \mathcal{P}(G)\ O(p)=R,\ T(p)=v.\]
A rooted graph will be a pair \((G,R)\) where \(G\) is a graph and \(R\in VG\) a root.

\begin{defn}\label{path-monoid-def}
    For any rooted graph \((G,R)\) the path monoid \(\mathcal{M}_p(G,R)\) is the graph monoid of \(\mathcal{P}(G,R)\), i.e. the commutative monoid generated by the set \(\mathcal{P}(G,R)\) and the relations
    \begin{align*}
        \forall p\in \mathcal{P}(G,R)\quad T(p)\in VG_{\text{reg}}\implies p=\sum_{e\in o^{-1}(T(p))}pe
    \end{align*}
	We will denote for any \(x,y \in  \mathcal{M}_p(G,R)\), \(x\preceq_{\mathcal{M}} y\) if \(\exists a\in \mathcal{M}_p(G,R),\ x=y+a\). We will call two elements \(x,y\) of the monoid independent (\(x\perp y\)) if
		\[(x\preceq_{\mathcal{M}} z)\land (y\preceq_{\mathcal{M}} z)\implies z=0\]
\end{defn}
We will drop the subscript \(\mathcal{M}\) most of the time when there is no possibility of confusion. Later we will show that being a prefix in the path monoid and in the space of paths only clash in very limited circumstances.

To better understand this monoid and show that each object of it has a standard form we will introduce the relation \(\rightarrow\) on the free commutative monoid \(\mathcal{FM}_p(G,R)\) generated by all the paths in \(\mathcal{P}(G,R)\), by setting
	\[\forall x\in \mathcal{FM}_p(G,R),p\in \mathcal{P}(G,R)_{\text{reg}},\ x+\sum_{e\in o^{-1}(T(p))} pe\rightarrow x+p.\]
We will denote by \(\overset{*}{\rightarrow}\) to be its transitive and reflective closure and \(\sim\) the equivalence relation generated by it. Note that we have \(\mathcal{M}_p(G,R)=\mathcal{FM}_p(G,R)/{\sim}\) using Newman's lemma\cite{Newman1942}\cite{Huet1980Confluence} we get for each element in the path monoid a minimal representative.

\begin{lemma}\label{newman-appl-lem}
	For each equivalence class \(z\in \mathcal{M}_p(G,R)=\mathcal{FM}_p(G,R)/{\sim}\) there exists a standard representative \(x_z\in \mathcal{FM}_p(G,R)\) of \(z\), i.e.  an element of the free group s.t.
		\[\forall y\in \mathcal{FM}_p(G,R),\ z=[y]_{\sim}\implies y \overset{*}{\rightarrow} x_z  \]
\end{lemma}
\begin{proof}
	We will be using the version of Newman's Lemma from \cite[Lemma 4]{Huet1980Confluence}. This means that we have to show that the relation \(\rightarrow\) is Noetherian (i.e. has no infinite chains) and locally confluent, i.e.
		\[ \forall x_1,x_2,y_1 \in \mathcal{FM}_p(G,R),\ x_1\leftarrow y_1\rightarrow x_2 \implies \exists y_2 \in \mathcal{FM}_p(G,R) x_1\overset{*}{\rightarrow} y_1\overset{*}{\leftarrow} x_2.\]
	To show that \(\rightarrow\) is noetherian we define the length \(|x|\) of an element \(x\in \mathcal{FM}_p(G,R)\) by setting
		\[|\sum_{p\in M}k_p p|=\sum_{p\in M}k_p|p|\in\mathbb{N},\]
	for any set of paths \(M\) and coefficients \(k_p\in \mathbb{N}\) for each \(p\in M\). From the definition of \(\rightarrow\) we see that
		\[x\rightarrow y \implies |x|> |y|,\]
	which makes infinite chains impossible and the relation noetherian.

	To show that \(\rightarrow\) is locally confluent we note that for any \(x_1,x_2,y_1 \in \mathcal{FM}_p(G,R)\) with \(x_1\leftarrow y_1\rightarrow x_2\) and \(x_1\neq x_2\) we have some \(w\in\mathcal{FM}_p(G,R)\) and paths \(p,q\in \mathcal{P}(G,R)\) s.t.
		\[y_1=w+p+q,\ x_1=w+\sum_{e\in o^{-1}(T(p))} pe+ q,\ x_2=w+p+\sum_{e\in o^{-1}(T(q))} qe.\]
	This means that 
		\[y_2:=w+\sum_{e\in o^{-1}(T(p))} pe+\sum_{e\in o^{-1}(T(q))} qe\]
	satisfies
		\[x_1\overset{*}{\rightarrow} y_1\overset{*}{\leftarrow} x_2.\]
	In the case where \(x_1=x_2\), \(y_2:=x_1\) satisfies the same property.

	Now using \cite[Lemma 2, Lemma 4]{Huet1980Confluence} we can see that any \(\rightarrow\) minimal element of a \(\sim\)- equivalence class is unique. So if we pick in each equivalence class an element with minimal length, this element is the unique \(\rightarrow\) minimal element in it, making it a standard representative as defined in the lemma.
\end{proof}



This will allow us to show that each element of this monoid can be written as a sum of paths in a finite set of pairwise independent paths. To this end we introduce some notation to help us when talking about sets of paths. 
\begin{defn}
	Take \(M,N\subseteq \mathcal{P}(G)\) then we can define
	\begin{itemize}
		\item \(M\) is \textbf{independent} if
			\[\forall p\neq q \in M,\ p\perp q.\]
		\item \(M\) \textbf{lies under} \(N\) if
		 	\[\forall q\in M,\exists p\in N\ p\preceq q\]
		we will write \(N\preceq M\).
		\item \(M\) \textbf{exactly covers} \(N\) if \(N\preceq M\) and for any path \(p\) with \(N\preceq \{p\}\) we have some \(r\) with \(M\preceq \{r\}\) and \(p\preceq r\).
 		\item \(M\) \textbf{covers} \(N\) if	
			\[\{p\in M\mid N\preceq \{p\}\}\]
		covers \(N\) exactly.
		\item We will also call \(M,N\) independent of each other, \(M\perp N\), if there exists no path \(p\) s.t. \(M\preceq p\succeq N\).
	\end{itemize}
	If a finite set of paths covers \(\{\varepsilon_R\}\) exactly and is independent we call it a \textbf{basis}.
    For some  \(q\in M\), \(M\) expanded by \(q\) is 
        \[M^q:=(M\setminus\{q\})\cup \{qe\mid e\in o^{-1}(T(q))\}\]
	and if we have a function \(f:M\to N\) that is terminus maintaining, i.e. \(\forall p \in M,\ T(p)=T(f(p))\), between two independent sets of paths we can define \(f^p:M^p\to N^{f(p)}\) for any \(p\in M\) by setting:
		\[f^p|_{M\setminus \{p\}}:=f|_{M\setminus \{p\}}\text{ and } \forall e\in o^{-1}(T(p)),\ f^p(pe):=f(p)e.\]
\end{defn}
For notational convenience we will write \(M\preceq p\) or \(p\preceq M\), instead of \(M\preceq \{p\}\) or \(\{p\}\preceq M\) respectively, for any \(p\in \mathcal{P}(G),M\subseteq\mathcal{P}(G)\)
We note that if \(M\) lies under or covers some set \(N\) so does \(M^q\). Additionally, if \(M\) independent then \(M^q\) is as well and if \(M\) is finite and \(q\) is regular then \(M^q\) is also finite. 

We will define for  any \(M,N\subseteq \mathcal{P}(G)\) 
	\[M_N:=\{p\in M\mid N\preceq p\}.\]

\begin{lemma}\label{lem1}
    For any \(p\in \mathcal{P}(G,R)\) and any finite \(M\subseteq \mathcal{P}(G,R)\) and coefficients \(k_q\in \mathbb{Z}_{>0}\) for each \(q\in M\) we have in the path monoid \(\mathcal{M}_p(G,R)\)
        \[p=\sum_{q\in M} k_q q,\]
    if and only if \(M\) is independent, \(k_q=1\) for each \(q\in M\) and \(M\) exactly covers some \(p_0\preceq p\), s.t. for any \(p_0\preceq p_1\precneq p\) we have \(|o^{-1}(T(p_1))|=1\). 
\end{lemma}
\begin{proof}
	We will first show the "only if" direction of the Lemma.
	We will first define \(p_{-}\) to be the shortest prefix of \(p\) s.t. for each \(p_-\preceq p_1\precneq p\) we have \(|o^{-1}(T(p_1))|=1\). It is the case that \(p_{-}=p\) in the path monoid, because if \(|o^{-1}(T(p_1))|=1\), we have \(p_1=p_1e\) for the unique edge s.t. \(o^{-1}(T(p_1))=\{e\}\). From the definition of \(\rightarrow\) it follows that \(p_{-}\) is a \(\rightarrow\) minimal element in \([p]_{\sim}\) and thus by \lemref{newman-appl-lem} it is the standard representative. 

	If we now note that for any independent set of paths \(M\) and any \(x\in \mathcal{FM}_p(G,R)\) 
		\[x\rightarrow \sum_{p\in M} p\implies \exists q\in M,\ x=\sum_{p\in M^q} p\]
	and \(M^q\) is independent and covers \(M\) exactly for any \(q\in M\). The lemma follows since for any \(y\in [p]_{\sim}\) we have \(y\overset{*}{\rightarrow} p_{-}\)

	For the proof of the converse we will show that for each independent \(M\subseteq \mathcal{P}(G,R)\) that covers \(p_{-}\) exactly we have
		\[p_{-}=\sum_{q\in M} q\]
	and since \(p=p_{-}\) in the graph monoid, the desired result follows. We will proceed by induction over the sum of lengths of paths in \(M\)
		\[\lVert M\rVert=\sum_{q\in M} |q|.\]
	We will begin with the case  \(\lVert M\rVert=|p_{-}|\), which implies that \(M=\{p_{-}\}\) since \(p_{-}\preceq M\). The desired property follows. 
	For the induction step we note that for any \(M\subseteq \mathcal{P}(G,R)\) and any regular \(q\in M\) we have 
		\[\lVert M\rVert<\lVert M^q \rVert\quad \text{  and  }\quad \sum_{r\in M} r=\sum_{s\in M^p} s \]
	in \(\mathcal{M}_p(G,R)\).  This means we will only have to show that for any independent \(M\neq \{p_{-}\}\) that covers \(p_{-}\) exactly we have another independent  \(N\subseteq\mathcal{P}(G,R)\) that covers \(p_{-}\) exactly and some regular \(q\in N\) s.t. \(M=N^{q}\). For this we choose \(q_0\in M\) to be an element of \(M\) with maximal length. Since \(M\neq \{p_{-}\}\) we have \(p_{-}\precneq q_0 \), which means that there exist some \(q\) with \(p_{-}\preceq q\) and \(q_0=qe_0\) for some \(e_0\in o^{-1}(T(q_0))\). We will show that for any \(e\in o^{-1}(T(q))\), \(qe\in M\), using that \(M\) covers \(p_{-}\) and is independent. For this we note that \(p_{-}\preceq qe\), and thus we have \(r_1\in M\) and \(r_2\in\mathcal{P}(G,R)\) with \(qe\preceq r_2\) and \(r_1\preceq r_2\), which means that \(qe\) and \(r_1\) are not independent, meaning that either \(qe\preceq r_1\) or \(r_1\precneq qe\).

	In the first case we have by maximality of \(q_0\)
		\[|q_0|=|qe_0|=|qe|\leq |r_1|\leq |q_0|\]
	and so \(|qe|=|r_1|\) meaning that \(qe=r_1\in M\). In the second case we have \(r_1\preceq q\) and thus \(r_1\precneq qe_0\in M\), contradicting the assumption that \(M\) is independent. So we must have \(qe\in M\), for any \(e\in o^{-1}(T(q))\) meaning that \(q\) must be regular as \(M\) is finite and if we define
		\[N:=(M\setminus \{qe\mid e\in o^{-1}(T(q)) \})\cup \{q\}\]
	this is an independent set that covers \(p_{-}\) exactly and \(M =N^{q}\).

\end{proof}
This lemma gives us an important corollary.
\begin{cor}
	For any finite independent sets \(M,N\subseteq \mathcal{P}(G,R)\) that both exactly cover some \(p\in\mathcal{P}(G,R)\) we have paths \(p_1,\cdots,p_n,q_1,\dots q_m\in\mathcal{P}(G,R)\) s.t.
		\[((M^{p_1})^{\cdots})^{p_n}=((N^{q_1})^{\cdots})^{q_n}.\]
\end{cor}
\begin{proof}
	Due to \lemref{lem1} we have
		\[\sum_{r\in M} q=p=\sum_{s\in N} s\]
	in the path monoid. Now, if we interpret the path monoid as a graph monoid, then due to \cite[Lemma 4.3]{ara2007} the corollary follows.
\end{proof}

To distinguish the paths \(p\) for which the condition in \lemref{lem1} simplifies to ``\(M\) covers \(p\) exactly '', we will refer to paths of the form \(qe\) where \(q\in \mathcal{P}(G)\), \(e\in EG\) are s.t.
	\[T(q)=o(e)\text{ and } |o^{-1}(T(q))|>1,\]
as \textbf{faithful}. Furthermore, we will call graphs \(G\) faithful, if for each \(p\in \mathcal{P}(G)\) there exist a faithful \(q\in \mathcal{P}(G)\) s.t. \(p\preceq q\).

\begin{lemma}\label{lem2}
    For \((G,R)\) an \(o\)-finite rooted graph and any \(x\in \mathcal{M}_p(G,R)\) and any basis \(B\) there exists an independent set \(M\) with \(B\preceq M\) and coefficients \(k_p\in \mathbb{Z}_{>0}\) for each \(p\in M\) s.t.
        \[x=\sum_{p\in M} k_p p\]
\end{lemma}
\begin{proof}
    We will write
        \[ x= \sum_{p\in N} l_p p\]
    for some not necessarily independent finite set of paths \(N\) and coefficients \(l_p\in \mathbb{Z}_{>0}\). Now if we take
        \[k:= \max\{|p|\mid p\in N\cup B\},\]
    we can define
        \[M_k=\{p\in \mathcal{P}(G,R)\mid |p|=k\}\cup\{p\in \mathcal{P}(G,R)_{\sing}\mid |p|\leq k\}.\]
    This is an independent set that lies under \(B\) s.t. for each \(p\in N\) the subset 
        \[M_{k,p}=\{q\in M_k\mid p\preceq q\}\subseteq M_k\]
    covers \(p\) exactly and is independent. So, using \lemref{lem1} we get
        \[ x=\sum_{p\in N} l_p p= \sum_{p\in N} l_p \big(\sum_{q\in M_{k,p}}q\big)=\sum_{q\in M_k}(\sum_{p\in N,p\preceq q}l_p)q\]
    showing the lemma.
\end{proof}
We can combine the above lemmas to get 

\begin{cor}
	For any faithful path \(p\in\mathcal{P}(G,R)\) and any \(x\in\mathcal{M}_p(G,R)\) s.t. \(p\preceq x\), there exists an independent set \(M\) that lies under \(p\), s.t. 
		\[x=\sum_{q\in M} q\] 
\end{cor}
\begin{proof}
	Since \(p\preceq x\), we have some \(a\in \mathcal{M}_p(G,R)\) s.t. \(p=x+a\). Using \lemref{lem2} we have some independent  sets \(M,N\subseteq \mathcal{P}(G,R)\) s.t.
		\[x=\sum_{m\in M} m,\quad a=\sum_{n\in N} n\]
	and thus
		\[p=x+a=\sum_{m\in M\setminus N} m+2\sum_{q\in M\cap N} q +\sum_{n\in N\setminus M}n.\]
	Using \lemref{lem1} this means that \(M\cap N=\emptyset\) and \(M\cup N\) is an independent set that covers \(p\) exactly, since \(p\) is faithful. This means that we must have \(p\preceq M\)
\end{proof}

The converse of the above corollary is a special case of a more general lemma.

\begin{lemma}
	For an \(o\)-finite rooted graph \((G,R)\) and any finite independent sets of paths \(M\preceq N\subseteq\mathcal{P}(G,R)\) we have
		\[\sum_{m\in M}m\preceq\sum_{n\in N}n.\]
\end{lemma}
\begin{proof}
	Since we have \(M\preceq N\) and \(N\) is independent, there is an independent set of paths \(C\) with \(N\subseteq C\) that covers \(M\) exactly. For each \(m \in M\) we can define
		\[C_m=\{c\in C\mid m\preceq c\}.\]
	Since \(C\) covers \(M\) exactly, \(C_m\) also covers \(m\) exactly and thus, using \lemref{lem1},
		\[\sum_{m\in M} m=\sum_{m\in M}\sum_{c\in C_p} c=\sum_{c\in C} c=\sum_{c\in C\setminus N}c +\sum_{n\in N}n.\]
	This shows that \(\sum_{m\in M}m\preceq\sum_{n\in N}n\).
\end{proof}

So in any o-finite graph, for any faithful path \(p\in \mathcal{P}(G,R)\) and any \(x\in \mathcal{M}_p(G,R)\) we have \(p\preceq x\) if and only
    \[x=\sum_{m\in M} m\]
for some independent set of paths \(M\), s.t. \(p\preceq M\). Additionally, if we have  \(p\preceq k\cdot x\) for some \(k\in \mathbb{Z}_{>0}, x\in \mathcal{M}_p(G,R)\), then we have either \(k=1\) or \(x=0\). Finally, for any \(x\perp y\in \mathcal{M}_p(G,R)\) if \(p\preceq x, p\preceq y\) we also have \(p\preceq x+y\).

We can connect the path monoid to the graph monoid (as defined in \defref{graph-monoid-defn}), by defining the terminus homomorphism
    \[T_{\mathcal{M}}:\mathcal{M}_p(G,R)\to \mathcal{M}_G\]
to be the unique homomorphism with \(T_{\mathcal{M}}(p)=T(p)\) for each path \(p\in \mathcal{P}(G,R)\).
\begin{lemma}
    For any \(\varepsilon_R\preceq x,y\in \mathcal{M}_p(G,R)\), with \(T_{\mathcal{M}}(x)=T_{\mathcal{M}}(y)\) we have independent sets \(M,N \subseteq \mathcal{P}(G,R)\) s.t.
        \[x=\sum_{p\in M} p,\quad y=\sum_{p\in N} p\]
    and there is a bijection \(\pi:M\to N\) that is terminus maintaining, i.e.
        \[\forall p\in M,\ T(\pi(p))=T(p)\]
\end{lemma}
\begin{proof}
    Take independent sets \(M_0,N_0\subseteq \mathcal{P}(G,R)\) s.t. 
        \[x=\sum_{p\in M_0} p,\quad y=\sum_{q\in N_0}q\]
    and note that we have
        \[\sum_{p\in M_0} T(p)=T_{\mathcal{M}}(x)=T_{\mathcal{M}}(y)=\sum_{q\in N_0}T(q)\]
    in the graph monoid. Now by \cite[Lemma 4.3]{ara2007} this means we can apply the defining relation of the graph monoid (in the expanding direction) to make the sums equal in the free monoid. Since applying the relation at \(T(p)\), results in the sum over \(M^p\), we have sets \(M_1,M_2,M_3,\dots, M_l\) with paths \(p_1,p_2,p_3,\dots, p_k\) s.t. \(M_{i-1}^{p_i}=M_{i}\) and sets \(N_1,N_2,N_3,\dots, N_k\) with paths \(q_1,q_2,q_3,\dots, q_k\) s.t. \(N_{i-1}^{q_i}=N_{i}\) that satisfy
        \[\forall 0\leq i\leq k,\ T_{\mathcal{M}}(x)=\sum_{p\in M_i} T(p)\ \text{and}\ T_{\mathcal{M}}(y)=\sum_{p\in N_i} T(p)\]
    and
        \[\sum_{p\in M_l} T(p)\equiv\sum_{p\in N_k} T(p) \]
    in the free monoid \(\mathcal{FM}_G\) generated by \(VG\). Thus, we must have a terminus maintaining bijection \(\pi:M_l\to N_k\), showing the lemma.
\end{proof}
If we define the cone based on \(x\in\mathcal{M}_p(G,R)\) to be
    \[\mathcal{M}_p(G,R)_x:=\{y\in \mathcal{M}_p(G,R)\mid x\preceq y\}.\]
We note that if we have a finite independent set of paths \(C\) s.t. \(x=\sum_{p\in C} p\) we can see that \(\mathcal{M}_p(G,R)_x\) is the submonoid of \(\mathcal{M}(G,R)\) generated by the paths \(q\) that lie under \(C\) (\(C\preceq q\)) and we can decompose it as 
    \[\mathcal{M}_p(G,R)_x=\bigoplus_{p\in C}\mathcal{M}_p(G,R)_p.\]    
So if we have a terminus maintaining bijection \(\pi: M\to N\), for \(M,N\) finite and independent, then it can be extended to an isomorphism \(\pi:\mathcal{M}_p(G,R)_{x}\to \mathcal{M}_p(G,R)_{y}\), where \(x:=\sum_{p\in M}p, y:=\sum_{q\in N}q\), that takes \(pr\) to \(\pi(p)r\) for any \(p\in M\) and \(r\in\mathcal{P}(G,T(p))\) (they will share a name for notational convenience). We will call the cone isomorphisms that arise this way \textbf{shifts}. Since different bijections of finite independent sets can give rise to the same shifts, we will call these bijections representatives of a shift.  If we introduce the \textbf{elementary shifts} to be the isomorphisms 
    \(\pi_{p,q}:\mathcal{M}_p(G,R)_{p}\to \mathcal{M}_p(G,R)_{q}\)
for any paths \(p,q\) with \(T(p)=T(q)\), that arise in this way from the simple bijection \(\pi:\{p\}\to\{q\}\) taking \(p\) to \(q\), we can write any shift with representative \(\pi:M\to N\) as
    \[\pi=\bigoplus_{p\in M} \pi_{p,\pi(p)}.\]
We note that the shifts form a groupoid that is the rooted version of the pseudogroup defined in \cite{ARKLINT_EILERS_RUIZ_2018}. If we consider their action on the boundary of the path space, shifts form a pseudogroup, so we will refer to the groupoid of shifts as the \textbf{shift pseudogroup} and denote it by \(\mathcal{S}(G,R)\). This will also avoid confusion with the shift groupoid discussed in \cite{Matsumoto_2014}, which is the groupoid of germs of the shift pseudogroup.

Now if we look at a homomorphism between two path monoids \(\Gamma:\mathcal{M}_p(G,R)\to\mathcal {M}_p(H,S)\) we can define the following.
\begin{defn}
    For rooted graphs \((G,R),(H,S)\), let \(\Gamma:\mathcal{M}_p(G,R)\to\mathcal {M}_p(H,S)\) be a monoid homomorphism, we will call it:
    \begin{itemize}
		\item \textbf{root-preserving}, if 
			\[\Gamma(\varepsilon_R)=\varepsilon_S\]
        \item \textbf{terminus-preserving}, if
            \[\forall x,y\in\mathcal{M}_p(G,R),\  T_{\mathcal{M}}(x)=T_{\mathcal{M}}(y)\implies T_{\mathcal{M}}(\Gamma(x))=T_{\mathcal{M}}(\Gamma(y))\]
        \item \textbf{shift-preserving}, if for each shift \(\pi\) in \(\mathcal{M}_p(G,R)\) there exists a shift \(\gamma\) s.t.
            \[\Gamma\circ \pi=\gamma\circ \Gamma.\]
    \end{itemize}
\end{defn}
Note that being shift-preserving implies being terminus-preserving.

Being root-preserving gives us the following
\begin{lemma}\label{weak-indep-lem}
	For any root-preserving homomorphism \(\Gamma:\mathcal{M}_p(G,R)\to\mathcal {M}_p(H,S)\) and any \(\varepsilon_R\preceq x,y\) we have	
		\(x\perp y \implies \Gamma(x)\perp \Gamma(y)\)
\end{lemma}
\begin{proof}
	We note that if we have \(x\perp y\) we also have \(\varepsilon_R\preceq x+y\) and thus \(\Gamma(\varepsilon_R)=\varepsilon_S\preceq\Gamma(x)+\Gamma(y)\) because of \(\Gamma\) being root-preserving, and thus \(\Gamma(x)\perp \Gamma(y)\).
\end{proof}
Proving the above lemma in the non-root-preserving case is possible, but a bit more involved. Since the path monoid homomorphisms we will look at will be root-preserving, so this version of the lemma will suffice.

Now we will look at non-injective homomorphisms.

	


	
\begin{lemma}\label{non-inj-lem}
	For any root-preserving homomorphism \(\Gamma:\mathcal{M}_p(G,R)\to \mathcal{M}_p(H,S)\) we have
	\begin{multline*}
		\forall \varepsilon_R\preceq x,y\in \mathcal{M}_p(G,R),\ \Gamma(x)=\Gamma(y)\implies\\
		 \exists z,a,b\in \mathcal{M}_p(G,R),\ x=z+a,\ y=z+b,\ \Gamma(a)=\Gamma(b)=0.
	\end{multline*}
\end{lemma}
\begin{proof}
	We will first assume that \(x\perp y\), in this case we also have \(\Gamma(x)\perp \Gamma(y)\), and thus we must have \(\Gamma(x)=\Gamma(y)=0\).
	
	If these elements are not independent we take two sets \(M_1,M_2\subseteq \mathcal{P}(G,R)\) with \(M_1\cup M_2\) independent, s.t. 
		\[x=\sum_{p\in M_1} p,\quad y=\sum_{q\in M_2} q.\]
	So by setting 
		\[ a:=\sum_{p\in M_1\setminus M_2} p,\ b:=\sum_{q\in M_2\setminus M_1} q \text{ and } z:=\sum_{r\in M_1\cap M_2} r  \]
	we get \(x=a+z\), \(y=b+z\) and \(a\perp b\). Since \(\Gamma(x)=\Gamma(y)\) we must also have \(\Gamma(a)=\Gamma(b)\) and so by the previous argument they are both equal to \(0\). 
\end{proof}
	

Using \lemref{non-inj-lem} we can define the image of any partial automorphism of a path monoid under a root-preserving homomorphism.
\begin{lemma}\label{image-shift-lem}
	For any \(u,w \in \mathcal{M}_p(G,R)_{\varepsilon_R}\), any isomorphism \(\pi:\mathcal{M}_p(G,R)_{u}\to\mathcal{M}_p(G,R)_w\) with
	 \[\forall u\preceq x,\ T_{\mathcal{M}}(\pi(x))=T_{\mathcal{M}}(x)\]
	and any root- and terminus-preserving homomorphism \(\Gamma:\mathcal{M}_p(G,R)\to\mathcal{M}_p(H,S)\), there exists a unique isomorphism
		\[\Gamma(\pi):\Gamma(\mathcal{M}_p(G,R)_{u})\to\Gamma(\mathcal{M}_p(G,R)_w)\]
	s.t.
		\[\forall x\in \mathcal{M}_p(G,R)_{u},\ \Gamma(\pi)(\Gamma(x))=\Gamma(\pi(x))\] 
\end{lemma}
\begin{proof}
	To prove this lemma we just have to show that
		\[\forall x,y\in \mathcal{M}_p(G,R)_{u},\ (\Gamma(x)=\Gamma(y)\implies \Gamma(\pi(x))=\Gamma(\pi(y))).\]
	For this we note that by \lemref{non-inj-lem}, since \(\varepsilon_R\preceq x,y\) and \(\Gamma(x)=\Gamma(y)\), there exists \(a,b,z\in\mathcal{M}_p(G,R)\) s.t. \(x=z+a\), \(y=z+b\) and \(\Gamma(a)=\Gamma(b)=0\). Since \(u\preceq x\preceq z\), \(u\preceq x\preceq a\) and \(u\preceq y\preceq b\), \(\pi\) is defined on \(a,b,z\). So we can observe
		\[\Gamma(\pi(x))=\Gamma(\pi(z))+\Gamma(\pi(a))\text{ and }\Gamma(\pi(x))=\Gamma(\pi(z))+\Gamma(\pi(b))\]
	meaning that we will just have to show that \(\Gamma(\pi(a))=\Gamma(\pi(b))=0\). Using the assumption that \(\Gamma\) is terminus preserving, we can factor it over \(T_{\mathcal{M}}\), which gives us the graph monoid homomorphism 
		\[\Gamma_T:\mathcal{M}_G\to\mathcal{M}_H\]
	that satisfies 
		\[\forall c\in\mathcal{M}_p(G,R),\ \Gamma_T(T_{\mathcal{M}}(c))=T_{\mathcal{M}}(\Gamma(c)).\]
	So we can calculate,
		\[T_{\mathcal{M}}(\Gamma(\pi(a)))=\Gamma_T(T_{\mathcal{M}}(\pi(a))=\Gamma_T(T_{\mathcal{M}}(a))=T_{\mathcal{M}}(\Gamma(a))=0\]
	which means that \(\Gamma(\pi(a))=0\). Analogously we can show that \(\Gamma(\pi(b))=0\) as well.
\end{proof}
Since this lemma applies is \(\pi\) is a shift, we can rephrase the condition for \(\Gamma\) to be shift-preserving as follows. \(\Gamma:\mathcal{M}_p(G,R)\to\mathcal{M}_p(H,S)\) is shift-preserving if and only if for each shift \(\pi:\mathcal{M}_p(G,R)_x\to\mathcal{M}_p(H,S)_y\) there exists a shift \(\gamma:\mathcal{M}_p(H,S)_{\Gamma(x)}\to\mathcal{M}_p(H,S)_{\Gamma(y)}\) s.t.
	\[\gamma|_{\text{Im}(\Gamma)}=\Gamma(\pi).\]

\begin{remark}\label{shift-homom-rmk}
	If we have a shift-preserving homomorphism \(\Gamma:\mathcal{M}_p(G,R)\to \mathcal{M}_p(H,S)\) we can pick (through a relatively involved process) for each \(p,q\in \mathcal{P}(G,R)\) with \(T(p)=T(q)\) shifts \(\gamma_{p,q}\) s.t.:
	\begin{itemize}
		\item \(\gamma_{p,q}|_{\text{Im}(\Gamma)}=\Gamma(\pi_{p,q})\)
		\item \(\gamma_{p,q}=\bigoplus_{e\in o^{-1}(T(p))}\gamma_{pe,qe}\)
		\item for each \(p_1,p_2,p_3\in\mathcal{P}(G,R)\) with \(T(p_1)=T(p_2)=T(p_3)\) we have
			\[\gamma_{p_1,p_1}=\text{id.},\ \gamma_{p_2,p_1}=\gamma_{p_1,p_2}^{-1}\text{ and }\gamma_{p_1,p_2}\circ\gamma_{p_2,p_3}=\gamma_{p_1,p_3}.\]
	\end{itemize}
	This will allow us to define a homomorphism of pseudogroups \(\Gamma_s:\mathcal{S}(G,R)\to \mathcal{S}(H,S)\) by mapping any shift with representative \(\pi:C_1\to C_2\)	to
		\[\Gamma_s(\pi)=\bigoplus_{p\in C_1} \gamma_{p,\pi(p)}.\]
	If the homomorphism \(\Gamma\) is surjective the choice of \(\gamma_{p,q}\)'s is unique, although \(\Gamma_s\) does not have to be surjective.
\end{remark}

\section{Constructing isomorphisms of path monoids via gluing diagrams}\label{gluing-diag-sec}
In this chapter we will define gluing diagrams in \defref{sec2-def1} and show how they describe homomorphisms between path monoids and shift pseudogroups. We will give a necessary and sufficient condition for these functions to be injective in \propref{inj-gluing-prop}. We will give more complex conditions in \lemref{surj-lem} and \thmref{shift-surj-thm} for the functions to be surjective that are only sufficient. In the case of unblocked diagrams defined in \defref{unblock-cover-lem}, these conditions are necessary and in the case of faithful graphs it is sufficient as well. Finally, we will define the operation of expansion on gluing diagrams \defref{exp-defn} and show that it does not change the homomorphism defined by it \propref{exp-no-chg-prop}.

\begin{defn}\label{multiset-defn}
	In order to define gluing diagrams we will be using multisets. A multiset \(\mathbf{M}\) with underlying set \(M\) is a set together with a function \(u:\mathbf{M}\to M\) assigning each element their underlying element. We will write elements \(\mathbf{a}\in \mathbf{M}\) bolded and their underlying elements unbolded i.e. \(u(\mathbf{a})=a\). 
	We will call two multisets \(\mathbf{M}_1,\mathbf{M}_2\) with the same underlying set \(M\) and associated functions \(u_1:\mathbf{M}_1\to M, u_2:\mathbf{M}_2\to M\)  equivalent if there exists a bijection \(f:\mathbf{M}_1\to\mathbf{M}_2\) s.t. \(u_1=u_2\circ f\). 
\end{defn}
We will freely switch between equivalent multisets as we will be only concerned by the underlying elements of the members of multisets and not their actual values. For any multiset \(\mathbf{M}\) with underlying set \(M\subseteq N\) we will write \(\mathbf{M}\subseteq_m N\)


We will denote by $\mathcal{FM}_G$ to be the free monoid generated by the vertices of $G$. We will identify the elements $x\in \mathcal{FM}_G$ with the unique (up to equivalence) multiset $\mathbf{M}_x\subseteq_m VG$ s.t. 
	\[x=\sum_{\mathbf{v}\in \mathbf{M}_x}v.\]
This way we can have elements in \(x\) and functions between arbitrary (multi)sets and $x$, and determine when they are injective or surjective. We will also be able to expand \(x\) along any \(\mathbf{v}\in x\) by setting 
	\[x^{\mathbf{v}}:=\sum_{\mathbf{w}\in \mathbf{M}_x\setminus\{\mathbf{v}\}} w +\sum_{e\in o^{-1}(v)} t(e).\]
As we can freely choose the representative for \(\mathbf{M}_{x^\mathbf{v}}\) we may assume \(\mathbf{M}_{x^\mathbf{v}}=\mathbf{M}_x\setminus\{\mathbf{v}\}\sqcup o^{-1}(v)\) where the underlying element of each \(e\in o^{-1}(v)\) is \(t(e)\).
We will denote for any $x\in\mathcal{FM}_G$ the path space rooted at \(x\), as:
	\[\mathcal{P}(G,x):=\{(p,\mathbf{v})\mid p\in\mathcal{P}(G), \mathbf{v}\in\mathbf{M}_x,\ O(p)=v\}.\]
We will write the element \(p_{\mathbf{v}}:=(p,\mathbf{v})\in\mathcal{P}(G,x)\) with \(O(p_{\mathbf{v}})=\mathbf{v}\in x\) and \(T(p_{\mathbf{v}})=T(p)\in VG\). For notational convenience we will simply write \(\varepsilon_{\mathbf{v}}:=(\varepsilon_v)_{\mathbf{v}}\). For any \(x,y\in \mathcal{FM}_G\) and any \(p_{\mathbf{v}}\in \mathcal{P}(G,x)\), \(q_{\mathbf{w}}\in \mathcal{P}(G,y)\) with \(T(p)=w\) (i.e. is the underlying object of \(O(q_{\mathbf{w}})\)), we can define 
	\[p_{\mathbf{v}}q_{\mathbf{w}}=(pq)_{\mathbf{v}}\in\mathcal{P}(G,x).\]
Analogously for \(q\in\mathcal{P}(G,T(p_{\mathbf{v}}))\) we have
	\[p_{\mathbf{v}}q=(pq)_{\mathbf{v}}\in\mathcal{P}(G,x).\]
We can thus define when paths in \(\mathcal{P}(G,x)\) are prefixes of each other
	\[p,q\in\mathcal{P}(G,x),\ p\preceq q \iff \exists r\in\mathcal{P}(G,T(p)),\ q=pr.\]
This allows us to define everything we have defined in \(\mathcal{P}(G,R)\) in \(\mathcal{P}(G,x)\) in an analogous way (including independent sets, bases and sets lying under or covering each other).
We can thus treat elements of \(\mathcal{P}(G,x)\) just as we did normal paths with exception that their origins are members of the multiset \(\mathbf{M}_x\) and not just simple vertices.
If we have a set of paths $M\subseteq \mathcal{P}(G,y)$ for some \(x\in\mathcal{FM}_G\), we call a function $f:M\to x$, terminus-maintaining if $T(p)$ is the underlying element of $f(p)$, for any $p\in M$. If \(M\) is an independent set we can extend this function naturally to \(f:\mathcal{P}(G,y)_M\to \mathcal{P}(G,x)\) that takes \(pr\) to \(r_{f(x)}\) for each \(p\in M\) and \(r\in \mathcal{P}(G,T(p))\) (this is only well-defined since \(M\) is independent). We note that the expansion of \(f\) is bijective if \(f\) is. Additionally, for any \(p\in M\) we can modify \(f\) to a function \(f^p:M^p\to x^{f(p)}\) by setting
	\[f^p|_{M\setminus \{p\}}:=f|_{M\setminus \{p\}}\text{ and }\forall e\in o^{-1}(T(p)),\ f^p(pe)=e\]
where we use the convention that \(\mathbf{M}_{x^{\mathbf{v}}}= \mathbf{M}_x\sqcup o^{-1}(v)\).

We can also take some set of paths \(N\subseteq\mathcal{P}(G,x)\) and define
	\[M\circ_f N=f^{-1}(N)\]
to be \(N\) glued to \(M\) along \(f\).
We note that, in the case of \(f\) being a bijection, if \(N\) is independent then \(M\circ_f N\) is as well and if \(N\) is a basis then \(M\circ_f N\) covers \(M\) exactly.

\begin{lemma}\label{lem1-sec2}
	For any independent set of paths $M$ with a bijection $\gamma: M\to\mathbf{M}_x$ and any basis $B\subseteq\mathcal{P}(G,x)$  we have:
	\[\sum_{p\in M}p=\sum_{q\in M\circ_{\gamma}B}q\]
\end{lemma}
\begin{proof}
	Since \(B\) is a basis \(M\circ_{\gamma}B=\gamma^{-1}(B)\) covers \(M\) exactly. Therefore, by \lemref{lem1} we must have
		\[\sum_{p\in M}p=\sum_{p\in M}\sum_{\substack{q\in M\circ_{\gamma}B\\ p\preceq q}}q=\sum_{q\in M\circ_{\gamma}B} q\]

\end{proof}
\begin{defn}\label{sec2-def1}
	A \textbf{gluing diagram} \(\mathfrak{G}=(x_v,C_{\varepsilon_R},\gamma_{\varepsilon_R},C_e,\gamma_e)_{e\in EG,v\in VG}\), connecting $(G,R)$ and $(H,S)$ consists of the following:
	\begin{itemize}
		\item For each $v\in VG$ some $ x_v\in\mathcal{FM}_H$.		
		\item A \textbf{starting basis} $C_{\varepsilon_R}\subseteq\mathcal{P}(H,S)$ together with a terminus maintaining bijection $\gamma_{\varepsilon_R}:C_{\varepsilon_R}\to x_{R}$.
		\item For each $e\in EG$ we have disjoint independent sets $C_e\subseteq\mathcal{P}(H,x_{o(e)})$ together with terminus maintaining bijections $\gamma_{e}:C_e\to x_{t(e)}$, such that
		    \[B_v:=\bigsqcup_{e\in o^{-1}(v)}C_{e}\subseteq \mathcal{P}(G,x_v)\]
		is a basis for each \(v\in VG\)
	\end{itemize}
	A structure that consists of everything above, except for the starting basis, will be called a \textbf{floating gluing diagram}.  
\end{defn}
For a visual representation of (floating) gluing diagram see \figref{fig1}.
We note that each such a diagram defines a (sometimes incomplete) tiling of the path space. We will also have to often pick concrete representatives of \(\mathbf{M}_{x_{t(e)}}\) to get concrete representatives of the \(\gamma_{e}\) for some \(e\in EG\) and while we can pick any representative individually the choice has to be consistent for all edges, to allow composition of the functions. 
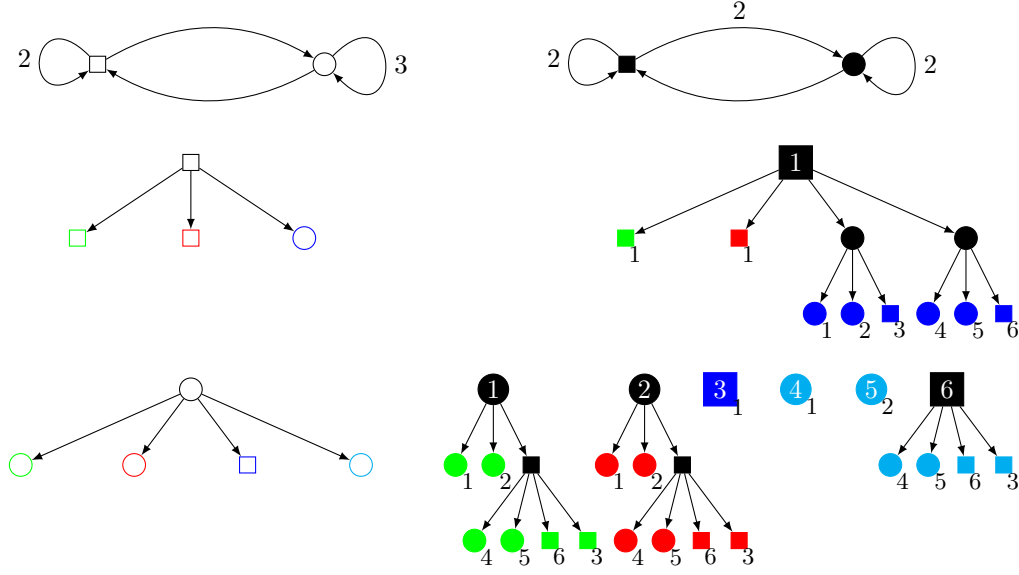
\begin{figure}
	\begin{tikzpicture}
		\node[big hollow square]     (x) at (-5,0)   {};
		\node[big hollow circle]   (y) at (-2,0)    {};

		\path[-latex]
		(x)     edge[bend left]                         node[above]             {} (y)
		edge[out=135, in=215, looseness=15]     node[left]              {2} (x)
		(y)     edge[bend left]                         node[below]             {}  (x)
		edge[out=45, in=315, looseness=15]      node[right]             {3} (y)
		;
		\node[big solid square]     (z) at (2,0)   {};
		\node[big solid circle]   (w) at (5,0)    {};

		\path[-latex]
		(z)     edge[bend left]                         node[above]             {2} (w)
		edge[out=135, in=215, looseness=15]     node[left]              {2} (z)
		(w)     edge[bend left]                         node[below]             {}  (z)
		edge[out=45, in=315, looseness=15]      node[right]            {2} (w)
		;
	\end{tikzpicture}
	\begin{tikzpicture}
		[edge from parent/.style={draw,-latex},
			level distance=10mm,
			level 1/.style={sibling distance=15mm},
			level 2/.style={sibling distance=5mm,},
			level 3/.style={sibling distance=3mm,},
			label distance=-2mm]
		\node[big hollow square] (r1) at (-4,0) {}
		child {node[big hollow green square] {}}
		child {node[big hollow red square] {}}
		child {node[big hollow blue circle] {}};
		\node[big solid square] (r2) at (4,0) {\textcolor{white}{1}}
		child {node[big solid green square, label={[font=\small,text=black]-10:$1$}] {}}
		child {node[big solid red square, label={[font=\small,text=black]-10:$1$}] {}}
		child {node[big solid circle] {}
				child {node[big solid blue circle, label={[font=\small,text=black]-10:$1$}] {}}
				child {node[big solid blue circle, label={[font=\small,text=black]-10:$2$}] {}}
				child {node[big solid blue square, label={[font=\small,text=black]-10:$3$}] {}}
			}
		child {node[big solid circle] {}
				child {node[big solid blue circle, label={[font=\small,text=black]-10:$4$}] {}}
				child {node[big solid blue circle, label={[font=\small,text=black]-10:$5$}] {}}
				child {node[big solid blue square, label={[font=\small,text=black]-10:$6$}] {}}
			};
		\node[big hollow circle] (r1) at (-4,-3) {}
		child {node[big hollow green circle] {}}
		child {node[big hollow red circle] {}}
		child {node[big hollow blue square] {}}
		child {node[big hollow cyan circle] {}};

		\node[big solid circle] (r2) at (0,-3) {\textcolor{white}{1} }
			child[sibling distance=5mm] {node[big solid green circle, label={[font=\small,text=black]-10:$1$}] {}}
			child[sibling distance=5mm] {node[big solid green circle, label={[font=\small,text=black]-10:$2$}] {}}
			child[sibling distance=5mm] {node[big solid square] {}
					child {node[big solid green circle, label={[font=\small,text=black]-10:$4$}] {}}
					child {node[big solid green circle, label={[font=\small,text=black]-10:$5$}] {}}
					child {node[big solid green square, label={[font=\small,text=black]-10:$6$}] {}}
					child {node[big solid green square, label={[font=\small,text=black]-10:$3$}] {}}
				};

		\node[big solid circle] (r3) at (2,-3) {\textcolor{white}{2} }
			child[sibling distance=5mm] {node[big solid red circle, label={[font=\small,text=black]-10:$1$}] {}}
			child[sibling distance=5mm] {node[big solid red circle, label={[font=\small,text=black]-10:$2$}] {}}
			child[sibling distance=5mm] {node[big solid square] {}
					child {node[big solid red circle, label={[font=\small,text=black]-10:$4$}] {}}
					child {node[big solid red circle, label={[font=\small,text=black]-10:$5$}] {}}
					child {node[big solid red square, label={[font=\small,text=black]-10:$6$}] {}}
					child {node[big solid red square, label={[font=\small,text=black]-10:$3$}] {}}
				};

		\node[big solid blue square, label={[font=\small,text=black]-10:$1$}] (r4) at (3,-3) {\textcolor{white}{3}};

		\node[big solid cyan circle, label={[font=\small,text=black]-10:$1$}] (r5) at (4,-3) {\textcolor{white}{4} };

		\node[big solid cyan circle, label={[font=\small,text=black]-10:$2$}] (r6) at (5,-3) {\textcolor{white}{5} };

		\node[big solid square] (r7) at (6,-3) {\textcolor{white}{6}}
		child[sibling distance=5mm] {node[big solid cyan circle, label={[font=\small,text=black]-10:$4$}] {}}
		child[sibling distance=5mm] {node[big solid cyan circle, label={[font=\small,text=black]-10:$5$}] {}}
		child[sibling distance=5mm] {node[big solid cyan square, label={[font=\small,text=black]-10:$6$}] {}}
		child[sibling distance=5mm] {node[big solid cyan square, label={[font=\small,text=black]-10:$3$}] {}};
	\end{tikzpicture}
	
	\caption{A floating gluing diagram connecting two 2-vertex graphs, the partitions are indicated by colors and the bijections by numbers. We can make it rooted by assigning the squares to be the roots of the graphs and setting the starting basis to be just the empty paths based on the black square.}\label{fig1}
\end{figure}

We will construct a homomorphism from $\mathcal{M}_p(G,R)$ to $\mathcal{M}_p(H,S)$ from a gluing diagram. For this we will recursively assign each path $p\in\mathcal{P}(G,R)$ an independent set of paths $C_p$ with a terminus-maintaining bijection $\gamma_p:C_p\to x_{T(p)}$. For this we first note that the gluing diagram already provides us with $C_{\varepsilon_R}$ and $\gamma_{\epsilon_R}$. If we have $C_p,\gamma_p$ we can define $C_{pe}$ for any $e\in o^{-1}(T(p))$ as
	\[C_{pe}=C_p\circ_{\gamma_p}C_e\]
and defined \(\gamma_{pe}:=\gamma_e\circ \gamma_p|_{C_{pe}}:C_{pe}\to x_{t(e)}\) where we view \(\gamma_p\) as a function from \( \mathcal{P}(G,R)_{C_p}\) to \(\mathcal{P}(G,x_{t(e)})\) as before. This construction also gives us a partition
\[C_p\circ_{\gamma_p}B_{T(p)}=\bigsqcup_{e\in o^{-1}(T(p))}C_{pe}.\]

So if we define the homomorphism $\Gamma:\mathcal{M}_p(G,R)\to\mathcal{M}_p(H,S)$ to be the one induced by
\[\forall p\in\mathcal{P}(G,R),\ \Gamma(p):=\sum_{q\in C_p} q,\]
it is well-defined since by \lemref{lem1-sec2}
\[\Gamma(p)=\sum_{q\in C_p} q= \sum_{q\in C_p\circ_{\gamma_p} B_{T(p)}} q=\sum_{e\in o^{-1}(T(p))}(\sum_{q\in C_{pe}} q)= \sum_{e\in o^{-1}(T(p))}\Gamma(pe).\]

\lemref{non-inj-lem} gives us the straightforward result
\begin{prop}\label{inj-gluing-prop}
	For any rooted graphs \((G,R),(H,S)\), a function \(\Gamma:\mathcal{M}_p(G,R)\to \mathcal{M}_p(H,S)\) described by a gluing diagram \(\mathfrak{G}=(x_v,C_{\varepsilon_R},\gamma_{\varepsilon_R},C_e,\gamma_e)_{e\in EG,v\in VG}\) is injective if and only if
		\[\forall v\in VG, x_v\neq 0.\]
\end{prop}


In contrast, the surjectivity of $\Gamma$ is not as easily determined. To see when this is the case we will characterize iterated bases tied to gluing diagrams. For this we will essentially repeat the above construction, but instead of starting from the starting basis we will start from \(C_{\varepsilon_v}:=\{\varepsilon_{\mathbf{w}}\mid \mathbf{w}\in x_v\}\subseteq \mathcal{P}(H,x_v)\).

Then we define recursively for each $r\in\mathcal{P}(G, v)$ and any $e\in o^{-1}(v)$
	\[C_{re}:=C_r\circ_{\gamma_r} C_e\]
and the function $\gamma_{re}:C_{re}\to x_{t(e)}$ defined by \(\gamma_{re}=\gamma_{e}\circ\gamma_{r}|_{C_{re}}\). By a straightforward inductive argument we can see that for any \(r,s\in\mathcal{P}(G)\) with \(T(r)=O(s)\) we have
	\[C_{rs}=C_r\circ_{\gamma_r} C_s,\]
which also means that 
	\[p\preceq q\implies C_p\preceq C_q,\ p\perp q \implies  C_p\perp C_q \text{ and for } C_p,C_q\neq \emptyset, C_p\perp C_q\implies p\perp q.\]

We can now define an \textbf{iterated basis} to be a set $B\subseteq\mathcal{P}(H,x_v)$ is of the form
  	\[B=\bigsqcup_{r\in B'}C_r,\]
where $B'\subseteq\mathcal{P}(G,v)$ is a basis. We note that for any \(p\in B'\), we have 
	\[B\circ_{\gamma_{p}}B_{T(p)}=\bigsqcup_{r\in (B')^p}C_r\]
and thus any iterated basis can be constructed by repeatedly  gluing \(B_w\)'s to \(\{\varepsilon_{\mathbf{v}}\mid \mathbf{v}\in x_v\}\).

It is somewhat notationally awkward that we have now doubly defined \(C_p\) for some \(p\in\mathcal{P}(G,R)\), first as a subset of \(\mathcal{P}(G,R)\) and later as a subset of \(\mathcal{P}(G,x_R)\). The former set can be achieved by gluing the latter to the starting basis, but they are technically different. To (hopefully) avoid confusion we will take care to always specify the larger space that the \(C_p\)'s we are discussing are subsets of.

As a first step to showing when \(\Gamma\) is surjective we will show when each \(q\in\mathcal{P}(H,S)\) can be covered by a union of \(C_p\)'s. For this we will need the notion of an element in \(x_v\) being blocked.

\begin{defn}
	Take a (floating) gluing diagram \(\mathfrak{G}=(x_v,C_{\varepsilon_R},\gamma_{\varepsilon_R}, C_e,\gamma_e)_{v\in VG,e\in EG}\) connecting \((G,R)\) to \((H,S)\). For any \(\mathbf{u},\mathbf{w}\in\sum_{v\in VG} x_v\) we say that \(\mathbf{u}\) \textbf{blocks} \(\mathbf{w}\) if there exists an \(e\in EG\) s.t. \(\varepsilon_{\mathbf{w}}\in C_e\) and \(\gamma_e(\varepsilon_{\mathbf{w}})= \mathbf{u}\). Additionally, we will set that for any \(v\in VG\) with \(o^{-1}(v)=\emptyset\) and amy \(\mathbf{u}\in x_v\), \(\mathbf{u}\) blocs itself.
	A tuple of elements of \(\sum_{v\in VG} x_v\), \((\mathbf{u}_1,\mathbf{u}_2, \dots ,\mathbf{u}_n)\) for some \(n\in\mathbb{Z}_{\geq 1}\) is a \textbf{blocking chain} of length \(n\) if
		\[\forall 1\leq i < n,\ \mathbf{u}_{i+1} \text{ blocks } \mathbf{u}_i.\]
	If we have \(\mathbf{u}_1=\mathbf{u}_n\) we  will call it a \textbf{blocking cycle}.
	We will call \(\mathbf{w}\) \textbf{unblocked} if it is not blocked by any element of \(\sum_{v\in VG} x_v\). In other words \(\mathbf{w}\in x_v\) is unblocked if \(\varepsilon_{\mathbf{w}}\notin B_v\).
	
	We will call the gluing diagram \textbf{unblocked} if for each blocking cycle \((\mathbf{u}_1,\mathbf{u}_2, \dots ,\mathbf{u}_n)\) we have \(o^{-1}(u_1)=\emptyset\).
\end{defn}

We note that if \(\mathbf{u}\) blocks \(\mathbf{w}\) they  must represent the same underlying vertex. So in an unblocked gluing diagram only elements with underlying vertices being sinks can be in a blocking cycle. Also note that if a vertex \(v\in VG\) has two loops adjacent to it then there is an element \(\mathbf{w}\in x_v\) that is unblocked. Additionally, if \(\mathbf{u}\in x_v\) is unblocked we must have \(\varepsilon_{\mathbf{u}}\not\in B_v\).

\begin{lemma}\label{unblock-cover-lem}
	For any gluing diagram \(\mathfrak{G}\) connecting \((G,R)\) and \((H,S)\), \(\mathfrak{G}\) is unblocked if and only if for any \(p\in\mathcal{P}(H,S)\), there exist a set of independent paths \(D_p\subseteq \mathcal{P}(G,R)\) s.t. \(\bigcup_{q\in D_p} C_q\) covers \(p\). 
\end{lemma}
\begin{proof}
	First we will show that the ''only if'' implication.

	For this we will first prove the lemma in the case where there exist a \(q_0\in\mathcal{P}(G,R)\), \(p_0\in C_{q_0}\) and \(e\in o^{-1}(T(p_0))\) s.t. \(p=p_0e\). In this case we take \(\mathbf{u}_0:=\gamma_{q_0}(p_0)\) since the underlying vertex of \(\mathbf{u}_0\) is \(T(p_0)\) which cannot be a sink, we have a blocking chain \((\mathbf{u}_0,\mathbf{u}_1,\dots,\mathbf{u}_{n})\) with \(\mathbf{u}_n\) unblocked. We will proceed by induction over \(n\), the length of this chain.
	
	In the case where \(n=0\), \(\mathbf{u}_0\) is unblocked and so  \(v:=T(q_0)\) is not a sink and \(\varepsilon_{\mathbf{u}_0}\not \in B_{v}\), which means that \(B_v\) covers \(e_{\mathbf{u}_0}\) and thus
		\[\gamma ^{-1}_{q_0}(B_v)=C_{q_0}\circ_{\gamma_{q_0}} B_v=\bigcup_{d\in o^{-1}(T(q_0))}C_{q_0d}\]
	covers \(\gamma^{-1}_{q_0}(e_{\mathbf{u}_0})=p_0e=p\).

	For the induction step we will assume that \(n>0\) and thus \(\mathbf{u}_1\) blocks \(\mathbf{u}_0\), which means that we have some \(d_0\in o^{-1}(T(q_0))\) s.t. \(\varepsilon_{\mathbf{u}_0}\in C_{d_0}\) and \(\gamma_{d_0}(\varepsilon_{\mathbf{u}_0})=\varepsilon_{\mathbf{u}_1}\) (when viewing it as a function to \(\mathcal{P}(H, x_{t(d_0)})\)). And thus 
		\[p_0=\gamma^{-1}_{q_0}(\varepsilon_{\mathbf{u}_0})=\gamma^{-1}_{q_0}(\gamma^{-1}_{d_0}(\varepsilon_{\mathbf{u}_1}))=\gamma^{-1}_{q_0d_0}(\mathbf{u}_1)\in C_{q_0d_0}.\]
	So we can shift our setting to \(p_0\in C_{q_0d_0}\) and \(\mathbf{u}_1=\gamma_{q_0d_0}(p_0)\). And since the maximal blocking chain \((\mathbf{u}_1,\dots,\mathbf{u}_{n})\) starting at \(\mathbf{u}_1=\gamma_{q_0d_0}(p_0)\), is shorter, by induction the lemma holds.

	Now for the general case we consider some arbitrary \(p\in\mathcal{P}(H,S)\). Firstly, we note that we can assume that \(C_{\varepsilon_R}\preceq p\), since otherwise \(C_{\varepsilon_R}\) would cover \(p\), as it is a basis. We can also assume there exists no \(q\in\mathcal{P}(G,R)\) s.t. \(p\in C_q\), since in this case \(C_q\) would cover \(p\). With these assumptions we can write \(p=p_0er\) for some \(q_0\in \mathcal{P}(G,R)\), \(p_0\in C_{q_0}\), \(e\in o ^{-1}(T(p_0))\) and \(r\in \mathcal{P}(G,t(e))\). We will proceed by induction over \(|r|\). For this we note that by our previous consideration we have an independent set \(D_{p_0e}\) s.t. \(C_0:=\bigcup_{q\in D_{p_0e}} C_q\) covers \(p_0e\). Since \(p_0e\preceq p\) we have some \(C_0\preceq q\) s.t. \(p\preceq q\) and thus some \(p_1\in C_0\) with \(p_1\preceq q\), which means either \(p\preceq p_1\) or \(p_1\precneq p\). In the first case \(\{q\in C_0\mid p\preceq q\}\) is not empty and this combined with \(C_0\) being independent means that \(C_0\) covers \(p\) as well.
	
	If \(p_1\precneq p\), we can write \(p=p_1ds\), for \(d\in o^{-1}(T(p_1))\) and \(s\in\mathcal{P}(G,t(d))\), and take \(q_1\in D_{p_0e}\), s.t. \(p_1\in C_{q_1}\). Since \(C_0\) is independent and covers \(p_0e\), \(p_0e\preceq p\) implies \(p_0e\preceq p_1\) and thus \(p_0\precneq p_1\). This means that we must have \(|s|<|r|\), so by induction we get the covering set from the lemma for \(p\).

	For the ''if'' implication we will take some blocking cycle \((\mathbf{u}_1,\mathbf{u}_2, \dots ,\mathbf{u}_n)\) with \(o ^{-1}(u_1)\neq \emptyset\). Take \(v_1,v_2,\dots, v_n\in VG\), \(e_1,e_2,\dots, e_n\in EG\) s.t. \(\mathbf{u}_i\in x_{v_i}\), \(\varepsilon_{\mathbf{u}_i}\in C_{e_i}\) and \(\gamma_{e_i}(\varepsilon_{\mathbf{u}_i})=\mathbf{u}_{i+1\mod n}\) for each \(1\leq i\leq n\).  We will show that for any \(q\in\mathcal{P}(G,R)\) with \(T(q)=v_i\) and any \(p\in C_q\) with \(\gamma_q(p)=\mathbf{u}_i\) we have \(p\in C_{qe_i}\) and \(\gamma_{qe_i}(p)=\mathbf{u}_{i+1\mod n}\).
	
	This is the case since	
		\[p=\gamma^{-1}_{p}(\varepsilon_{\mathbf{u}_{i}})=\gamma^{-1}_{p}(\gamma^{-1}_{e_i}(\varepsilon_{\mathbf{u}_{i+1\mod n}}))=\gamma^{-1}_{pe_i}(\mathbf{u}_{i+1\mod n})\in C_{qe_i}.\]
	If we apply this result inductively this means that for any \(N>0\) there exists some \(q_N\in\mathcal{P}(G,R)\) s.t. \(p\in C_q\) and \(|q_N|\geq N\).

	Now, if we take some \(d\in o^{-1}(T(p))=o ^{-1}(u_1)\neq \emptyset\), we note that if there would exist \(C_q\)'s that cover \(pd\) as in our lemma we would have some \(\tilde{q}\in \mathcal{P}(G,R)\) with some \(\tilde{p}\in C_{\tilde{q}}\) s.t. \(pd\preceq \tilde{p}\). Now take some \(q\in\mathcal{P}(G,R)\) with \(|q|\geq |\tilde{q}|\) and \(p\in C_q\), since \(q\) is at least as long as \(\tilde{q}\) we have either \(\tilde{q}\preceq q\) or \(\tilde{q}\perp q\). In the first case we would have
		\[\tilde{p}\in C_{\tilde{q}}\preceq C_q\ni p\]
	which is a contradiction since \(p\precneq pd\preceq \tilde{p}\). In the other case we have by \lemref{weak-indep-lem}
		\[\tilde{q}\perp q \implies \Gamma(q)=\sum_{r\in C_q}r\perp \sum_{s\in C_{\tilde{q}}}s=\Gamma(\tilde{q})\]
	(where \(\Gamma\) is the path-monoid homomorphism described by the gluing diagram), which is also a contradiction since we have	
		\[\Gamma(q)=\sum_{r\in C_q}r\preceq p\preceq \tilde{p}\text{ and } \Gamma(\tilde{q})=\sum_{s\in C_{\tilde{q}}}s\preceq \tilde{p}.\]
\end{proof}

To show when an unblocked gluing diagram produces a surjective function we will have to show that each $\mathcal{P}(G,x_v)$ has a \textbf{splitting} iterated basis.
\begin{defn}\label{split-basis}
	For any $x\in\mathcal{FM}_G$ and any $\mathbf{v}\in \mathbf{M}_x$ we call a partitioned basis
	\[B=\bigsqcup_{i\in I} C_i\subseteq \mathcal{P}(G,x) \]
	\(\mathbf{v}\)-\textbf{splitting} if we have a subset $J\subseteq I$ s.t. $\bigsqcup_{j\in J}C_j$ covers \(\varepsilon_{\mathbf{v}}\) exactly.
\end{defn}

We will first show that 
\begin{lemma}\label{splitting-covering-lem}
	If an unblocked gluing diagram has a \(\mathbf{u}\)-splitting iterated basis of $\mathcal{P}(H,x_v)$ for each $v\in VG$ and each \(\mathbf{u}\in x_v\), then for each \(p\in \mathcal{P}(H,S)\) there exist an independent set \(D_p\subseteq \mathcal{P}(G,R)\) s.t.
		\[\bigcup_{q\in D_p} C_q\]
	covers \(p\) exactly.

	Conversely, for each \(v\in VG\) and each \(\mathbf{u}\in x_v\), the gluing diagram has an \(\mathbf{u}\)-splitting iterated basis if there exists a \(q\in \mathcal{P}(G,R)\), with \(T(q)=v\) and a \(p\in C_q\) with \(\gamma_q(p)=\mathbf{u}\), together with an independent set \(D_p\subseteq \mathcal{P}(G,R)\) s.t.
		\[\bigcup_{q\in D_p} C_q\]
	covers \(p\) exactly.
\end{lemma}
\begin{proof}
	For the first part, we will first assume that \(p\in C_q\) for some \(q\in\mathcal{P}(G,R)\). In this case we note that since we have \(\gamma_q(p)\)-splitting basis of \(\mathcal{P}(G,x_{T(q)})\), there exists an independent set \(D_{T(p)}\subseteq \mathcal{P}(G,T(p))\) s.t. \(\bigcup_{r\in D_{T(q)}} C_r\subseteq\mathcal{P}(G,x_{T(p)})\) covers \(\varepsilon_{\gamma_{q}(p)}\) exactly and thus
		\[\gamma^{-1}_q(\bigcup_{r\in D_{T(q)}} C_r)=\bigcup_{r\in D_{T(q)}} C_q\circ_{\gamma_q}C_r=\bigcup_{r\in D_{T(q)}} C_{qr}\]
	covers \(\gamma^{-1}_q(\varepsilon_{\gamma_{q}(p)})=p\) exactly. 

	For a general \(p\in \mathcal{P}(H,S)\), by \lemref{unblock-cover-lem}, we have some \(\tilde{D}_p\subseteq\mathcal{P}(G,R)\) s.t. 
		\[C=\bigcup_{q\in \tilde{D}_p} C_q\]
	covers \(p\), which means that
		\[\bigcup_{\substack{\tilde{p}\in C\\p\preceq\tilde{p}}}\bigcup_{q\in D_{\tilde{p}}}C_q,\]
	with \(D_{\tilde{p}}\) being as in the lemma (which exists since \(\tilde{p}\in C_{\tilde{q}}\), for some \(\tilde{q}\in \mathcal{P}(G,R)\)), covers \(p\) exactly.

	For the second part take \(v\in VG\) and \(\mathbf{u}\in x_v\), s.t. there exists \(q\in\mathcal{P}(G,R)\), with \(T(q)=v\) and some \(p\in C_q\) with \(\gamma_q(p)=\mathbf{u}\) together with some independent \(D_p\subseteq \mathcal{P}(G,R)\) s.t.
		\[\bigcup_{\tilde{q}\in D_p} C_{\tilde{q}}\]
	covers \(p\) exactly. We also assume that \(\forall \tilde{q}\in D_p,\ C_{\tilde{q}}\neq \emptyset\). This means especially that for any \(\tilde{q}\in D_p\), \(C_q\preceq C_{\tilde{q}}\), which means that \(q\) and \(\tilde{q}\) are not independent. Additionally, since any expansion of \(D_p\) also satisfies the above condition, we may assume that for each \(\tilde{q}\in D_q\) we have \(\tilde{q}\not\preceq p\). This means we have to have \(q\preceq D_p\), and so we can write
		\[\bigcup_{\tilde{q}\in D_p} C_{\tilde{q}}=\bigcup_{\substack{r\in \mathcal{P}(G,T(q))\\ qr\in D_p}} C_{qr}=\bigcup_{\substack{r\in \mathcal{P}(G,T(q))\\ qr\in D_p}} C_q\circ_{\gamma_q}C_r=\gamma^{-1}_q(\bigcup_{\substack{r\in \mathcal{P}(G,T(q))\\ qr\in D_p}} C_r)\]
	which means that \(\bigcup_{r\in \tilde{D}_p} C_r\) covers \(\varepsilon_{\gamma_q(p)}=\varepsilon_{\mathbf{u}}\) exactly, where 
		\[\tilde{D}_p:=\{r\in \mathcal{P}(G,T(q))\mid qr\in D_p\}\]
	is an independent set. So for any basis \(\tilde{B}\subseteq\mathcal{P}(G,T(q))=\mathcal{P}(G,v)\) with \(\tilde{D}_p\subseteq \tilde{B}\) the iterated basis 
		\[\bigcup_{r\in \tilde{B}} C_r\subseteq\mathcal{P}(G,x_v)\]
	is 	\(\mathbf{u}\)-splitting.
\end{proof}

Finally, we can use the above lemmas to show when a path-monoid homomorphism described by a gluing diagram is surjective. 

\begin{prop}\label{surj-lem}
	Any unblocked gluing diagram connecting $(G,R)$ to $(H,S)$ has an \(\mathbf{u}\)-splitting iterated basis of $\mathcal{P}(G,x_v)$ for each $v\in VG$ and each \(\mathbf{u}\in x_v\), if and only if it describes a surjective terminus preserving homomorphism $\Gamma:\mathcal{M}_p(G,R)\to\mathcal{M}_p(H,S)$. 

	Additionally, if \(H\) is faithful, then if \(\Gamma\) is surjective, the gluing diagram \(\mathfrak{G}\) is unblocked.

\end{prop}
\begin{proof}
	The ''only if'' part of the lemma follows form \lemref{splitting-covering-lem} since for any \(p\in\mathcal{P}(H,S)\) and any independent \(D_p\subseteq \mathcal{P}(G,R)\) s.t. \(\bigcup_{q\in D_p} C_q\) covers \(p\) exactly and thus by \lemref{lem1} we have	
		\[p=\sum_{q\in D_p}\sum_{\tilde{p}\in C_q} \tilde{p}=\sum_{q\in D_p}\Gamma(q)=\Gamma(\sum_{q\in D_p}q)\in\text{Im}(\Gamma).\]
	This means that \(\text{Im}(\Gamma)\) contains all paths, making \(\Gamma\) surjective.

	For the converse, take some $v\in VG$ and some \(\mathbf{u}\in x_v\), since \(R\) is a root we have some \(q\in \mathcal{P}(G,R)\) with \(T(q)=v\) and thus some \(p\in C_q\) with \(\gamma_q(p)=\mathbf{u}\). Since \(\Gamma\) is root-preserving and surjective we have some finite independent \(D_p\subseteq\mathcal{P}(G,R)\) s.t.
		\[p=\Gamma(\sum_{r\in D_p} r)=\sum_{r\in D_p}\sum_{s\in C_r} s.\]
	Since this equality is not changed when expanding \(D_p\), we may assume that \(D_p\not\preceq q\) and thus \(\bigcup_{r\in D_p} C_r\not\preceq C_q\), which implies that \(\bigcup_{r\in D_p} C_r\not\preceq p\) (using the fact that the \(C_{\tilde{q}}\)'s either lie under or are independent of each other). So, by \lemref{lem1}, \(\bigcup_{r\in D_p} C_r\) covers \(p\) exactly, meaning that, by \lemref{splitting-covering-lem}, we have a \(\mathbf{u}\)-splitting iterated basis.

	Now if we assume that \(H\) is faithful, we note that every path in \(\mathcal{P}(H,S)\) is covered by finitely many faithful paths. So we only have to show \lemref{unblock-cover-lem} for faithful paths. If we take some faithful \(p\in\mathcal{P}(H,S)\), we have since  \(\Gamma\) is root-preserving and surjective we have some finite independent \(D_p\subseteq\mathcal{P}(G,R)\) s.t.
		\[p=\Gamma(\sum_{r\in D_p} r)=\sum_{r\in D_p}\sum_{s\in C_r} s.\]
	Which means by \lemref{lem1}, that \(\bigcup_{r\in D_p}C_r\) covers \(p\) (exactly). So by \lemref{unblock-cover-lem}, \(\mathfrak{G}\) is unblocked.

\end{proof}
In the light of this lemma we will be calling unblocked (floating) gluing diagrams that have \(\mathbf{u}\)-splitting iterated bases of \(\mathcal{P}(G,x_v)\) for each \(v\in VG, \mathbf{u}\in x_v\), \textbf{surjective}.

We will now transition to looking at what happens to shifts when a homomorphism described by a gluing diagram is applied to it.

\begin{prop}\label{shift-pres-lem}
	Any \(\Gamma:\mathcal{M}_p(G,R)\to \mathcal{M}_p(H,S)\) constructed via a gluing diagram is shift-preserving. Specifically, for each \(p,q\in\mathcal{P}(G,R)\), with \(T(p)=T(q)\) the shift \(\gamma_{p,q}\) defined by the representative \(\gamma^{-1}_q\circ\gamma_{p}:C_p\to C_q\) satisfies
		\[\Gamma(\pi_{p,q})=\gamma_{p,q}|_{\text{Im}(\Gamma)}.\] 
\end{prop}
\begin{proof}
	We note that since any shift is a sum of elementary shifts it suffices to show the second part of the lemma. For this we will just have to show that $\gamma_{p,q}(C_{pr})=C_{qr}$ for any $r\in \mathcal{P}(G,T(p))$. This is implied by $\gamma_{p,q}|_{C_{pr}}=\gamma_{pr,qr}$, which can be inductively shown by proving that $\gamma_{p,q}|_{C_{pe}}=\gamma_{pe,qe}$ for any $e\in o^{-1}(T(p))$. To show this take some $s\in C_e$ and $r\in C_p$, s.t. $rs=\gamma_p^{-1}(s)\in C_{pe}$, then $\gamma_{p,q}(rs)=\gamma_{p,q}(r)s$, but since $\gamma_{pe}(rs)=\gamma_e(s)=\gamma_{qe}(\gamma_{p,q}(r)s)$ we also must have $\gamma_{pe,qe}(rs)=\gamma_{p,q}(r)s=\gamma_{p,q}(rs)$.
\end{proof}




We note that the \(\gamma_{p,q}\)'s from the above proposition satisfy the conditions from \rmkref{shift-homom-rmk} giving us a canonical homomorphism between the shift pseudogroups.

To show when an isomorphism constructed by a gluing diagram induces an isomorphism of shift-pseudogroups (we'll call such a morphism \textbf{shift-surjective}), we need to first introduce a few concepts.
\begin{defn}\label{internal-defn}
	For any \(x\in \mathcal{FM}_G\) and any basis \(B\subseteq \mathcal{P}(G,x)\), a path \(p\in \mathcal{P}(G,x)\) is \textbf{internal} if \(B\not\preceq p\).
\end{defn}
We note that each internal element \(p\) of a basis \(B\), has some \(q\in B\) with \(p\preceq q\). Since any basis is finite by definition, and any path has only finitely many prefixes, \(B\) only has finitely many internal paths. 
\begin{defn}\label{defn-enable}
	For any (floating) gluing diagram \(\mathfrak{G}\) connecting \((G,R)\) and \((H,S)\), any \(v,w\in VG\) and any \(p\in\mathcal{P}(H,x_v),q\in\mathcal{P}(H,x_w)\) with \(T(p)=T(q)\), we say that \(\mathfrak{G}\) \textbf{enables} the shift from \(p\) to \(q\) if there exist independent sets \(D_p\subseteq \mathcal{P}(G,v), D_q\subseteq \mathcal{P}(G,w)\) s.t.
		\[C_1:=\bigcup_{r\in D_p} C_r\subseteq \mathcal{P}(H,x_v)\text{ and }C_2:=\bigcup_{s\in D_q} C_s\subseteq \mathcal{P}(H,x_w)\]
	exactly cover \(p\) and \(q\) resp., and there exists a terminus-maintaining bijection \(\nu_{p,q}:D_q\to D_q\) s.t.
		\[ \forall r\in\mathcal{P}(H ,T(p))\forall s\in D_p,\ pr\in C_{s}\implies qr\in C_{\nu_{p,q}(q)}\text{ and } \gamma_{s}(pr)=\gamma_{\nu_{p,q}(s)}(qr).\]

	Analogously we can define for any (non-floating) gluing diagram \(\mathfrak{G}\) connecting \((G,R)\) and \((H,S)\), any \(p,q\in\mathcal{P}(H,S),\) with \(T(p)=T(q)\), we say that \(\mathfrak{G}\) \textbf{enables} the shift from \(p\) to \(q\) if there exist independent sets \(D_p, D_q\subseteq \mathcal{P}(G,R)\) s.t.
		\[C_1:=\bigcup_{r\in D_p} C_r\subseteq \mathcal{P}(G,R)\text{ and }C_2:=\bigcup_{s\in D_q} C_s\subseteq \mathcal{P}(G,R)\]
	exactly cover \(p\) and \(q\) resp., such that there exists a terminus-maintaining bijection \(\nu_{p,q}:D_q\to D_q\) s.t.
		\[ \forall r\in\mathcal{P}(H ,T(p))\forall s\in D_p,\ pr\in C_{s}\implies  qr\in C_{\nu_{p,q}(q)}\text{ and }\gamma_{s}(pr)=\gamma_{\nu_{p,q}(s)}(qr).\]

\end{defn}
\begin{remark}\label{exp-enable-rmk}
	We note that the above definition is stable under expansion, i.e. for any \(D_p,D_q,\nu_{p,q}\) that enable a shift from \(p\) to \(q\), then for any regular \(r\in D_p\), the sets \(D^r_p,D^{\nu_{p,q}(r)}_p\) together with the bijection \(\nu^r_{p,q}:D^r_p\to D^{\nu_{p,q}(r)}_p\) also enable the shift.
\end{remark}

We first show that a shift being enabled by a gluing diagram means that it is in the image of the homomorphism that this gluing diagram describes.
\begin{lemma}\label{enable-image-lem}
	For any surjective gluing diagram \(\mathfrak{G}\) describing the homomorphism \(\Gamma:\mathcal{M}_p(G,R)\to \mathcal{M}_p(H,S)\) and any \(p,q\in \mathcal{P}(H,S)\) with \(T(p)=T(q)\), \(\mathfrak{G}\) enables the shift from \(p\) to \(q\) if and only if there exists a shift \(\gamma_{p,q}\) in \(\mathcal{M}_p(G,R)\) s.t. \(\Gamma(\gamma_{p,q})=\pi_{p,q}\). 
\end{lemma}
\begin{proof}
	For the first implication take \(D_p, D_q\subseteq \mathcal{P}(G,R)\) and \(\nu_{p,q}:D_p\to D_q\) as in \defref{defn-enable} and set
		\(\gamma_{p,q}:=\bigoplus_{\tilde{p}\in D_p} \pi_{\tilde{p},\nu_{p,q}(\tilde{p})}.\)
	To show that \(\Gamma(\gamma_{p,q})=\pi_{p,q}\), we will first note that since \(\gamma_{p,q}: \mathcal{M}_p(G,R)_x\to \mathcal{M}_p(G,R)_y\) where
		\[x=\sum_ {\tilde{p}\in D_p}\tilde{p}\text{ and }y=\sum_ {\tilde{q}\in D_q}\tilde{q},\]
	and 
		\[\Gamma(x)=\sum_ {\tilde{p}\in D_p}\Gamma(\tilde{p})=\sum_ {\tilde{p}\in D_p}\sum_{r\in C_{\tilde{p}}} r= p\text{ and }\Gamma(y)=\sum_ {\tilde{q}\in D_q}\Gamma(\tilde{q})=\sum_ {\tilde{q}\in D_q}\sum_{r\in C_{\tilde{q}}} r= q\]
	by \defref{defn-enable}, we have \(\Gamma(\gamma_{p,q}):\mathcal{M}_p(H,S)_p\to \mathcal{M}_p(H,S)_q\). Now we will take some \(r\in\mathcal{P}(H,T(p))\) s.t. \(pr\in C_{\tilde{p}}\) for some \(\tilde{p}\in D_{p}\), then by \defref{defn-enable} and \lemref{shift-pres-lem} we have
		\[\Gamma(\gamma_{p,q})(pr)=\Gamma(\pi_{\tilde{p},\nu_{p,q}(\tilde{p})})(pr)=\gamma^{-1}_{\nu_{p,q}(\tilde{p})}(\gamma_{\tilde{p}}(pr))=qr.\]
	This means that \(\Gamma(\gamma_{p,q})\) and \(\pi_{p,q}\) agree on \(\bigcup_{\tilde{p}\in D_p}C_{\tilde{p}}\), which covers \(p\) exactly. So, as they are both shifts with domain \(\mathcal{M}_p(H,S)_p\), they must be equal.

	For the converse, first note that since \(\mathfrak{G}\) is unblocked we have for each \(p\in\mathcal{P}(H,S)\) a finite independent set \(M_p\subseteq \mathcal{P}(G,R)\), s.t. \(\bigcup_{r\in M_p} C_r\) covers \(p\). Now we will assume that  we have a shift \(\gamma_{p,q}\) s.t. \(\Gamma(\gamma_{p,q})=\pi_{p,q}\), we can write it as 
		\[\gamma_{p,q}=\bigoplus_{\tilde{p}\in D_p} \pi_{\tilde{p},\nu_{p,q}(\tilde{p})}\]
	for some independent sets \(D_p,D_q\subseteq\mathcal{P}(G,R)\) and some terminus-maintaining bijection \(\nu_{p,q}:D_p\to D_q\). Since for any \(\tilde{D}_p\) that covers \(D_p\) exactly, \(\tilde{D}_p, \gamma_{p,q}(\tilde{D}_p), \tilde{\nu}_{p,q}:=\gamma_{p,q}|_{\tilde{D}_p}\) satisfy the same property, due to \rmkref{exp-enable-rmk}, we may assume that for any \(r\in M_p\) \(D_p\not\preceq r\) and for any \(s\in M_q\) \(D_q\not\preceq s\). Thus, we have \(\bigcup_{\tilde{p}\in D_p}C_{\tilde{p}}\not \preceq p\) and \(\bigcup_{\tilde{q}\in D_q}C_{\tilde{q}}\not \preceq q\). 
	Since \(\pi_{p,q}\) is the image of \(\gamma_{p,q}\) we must have	
		\[p=\Gamma(\sum_{\tilde{p}\in D_p}\tilde{p})=\sum_{\tilde{p}\in D_p} \sum_{r\in C_{\tilde{p}}} r \text{ and }q=\Gamma(\sum_{\tilde{q}\in D_q}\tilde{q})=\sum_{\tilde{q}\in D_q} \sum_{s\in C_{\tilde{q}}} s.\]
	So due to \lemref{lem1} and our above assumptions about \(D_p\) and \(D_q\), \(\bigcup_{\tilde{p}\in D_p} C_{\tilde{p}}\) and \(\bigcup_{\tilde{q}\in D_q} C_{\tilde{q}}\) exactly cover \(p\) and \(q\) respectively. Additionally, for any \(r\in\mathcal{P}(H,T(p))\) s.t. \(pr\in C_{\tilde{p}}\) for some \(\tilde{p}\in D_{p}\) we have, using \lemref{shift-pres-lem},
		\[qr=\pi_{p,q}(pr)=\Gamma(\gamma_{p,q})(pr)=\gamma^{-1}_{\nu_{p,q}(\tilde{p})}(\gamma_{\tilde{p}}(pr)).\]
	This means that \(qr\in C_{\nu_{p,q}(\tilde{p})}\) and 
		\[\gamma_{\nu_{p,q}(\tilde{p})}(qr)=\gamma_{\tilde{p}}(pr).\]
	So \(\mathfrak{G}\) enables the shift from \(p\) to \(q\).
\end{proof}

We can also connect a shift being enabled between paths in \(\mathcal{P}(H,S)\) with shifts between paths in \(\bigsqcup_{v\in VG}\mathcal{P}(H,x_v)\).
\begin{lemma}\label{shift-translation-lem}
	For any gluing diagram \(\mathfrak{G}\), any \(p_0,q_0\in \mathcal{P}(H,S)\) with \(T(p_0)=T(q_0)\) and some \(p_1,q_1\in \mathcal{P}(G,R)\) s.t. \(C_{p_1}\preceq p_0\) and \(C_{q_1}\preceq q_0\), then \(\mathfrak{G}\) enables the shift from \(p_0\) to \(q_0\), if and only if it enables the shift from \(\gamma_{p_1}(p_0)\) to \(\gamma_{q_1}(q_0)\).
\end{lemma}
\begin{proof}
	Define \(p_2:=\gamma_{p_1}(p_0)\) and \(q_2:=\gamma_{q_1}(q_0)\).

	For the ''if'' implication, we take \(D_{p_2},D_{q_2},\nu_{p_2,q_2}\) as in \defref{defn-enable}. We will show that the sets
	\begin{align*}
		D_{p_0}:&=\{p_1r\mid r\in D_{p_2}\}\\
		D_{q_0}:&=\{q_1r\mid r\in D_{q_2}\}
	\end{align*}
	together with the terminus-maintaining bijection
		\[\nu_{p_0,q_0}:\begin{cases}
							D_{p_0}\to D_{q_0}\\
							p_1r\mapsto q_1\nu_{p_2,q_2}(r)
						\end{cases}\]
	satisfies \defref{defn-enable}. For this we note that 
		\[\bigcup_{p\in D_{p_0}} C_p=\bigcup_{r\in D_{p_2}} C_{p_1r}=\bigcup_{r\in D_{p_2}} C_{p_1}\circ_{\gamma_{p_1}} C_r=C_{p_1}\circ_{\gamma_{p_1}}\bigcup_{r\in D_{p_2}} C_r=\gamma^{-1}_p(\bigcup_{r\in D_{p_2}} C_r)\]
	covers  \(p_0=\gamma ^{-1}_{p_1}(p_2)\) exactly. With an analogous consideration we can show that \(\bigcup_{q\in D_{q_0}} C_{q}\) covers \(q_0\) exactly. Now we can write any \(p\in D_{p_0}\) as \(p=p_1r\) for some unique \(r\in D_{p_2}\), and thus we have for each \(s\in \mathcal{P}(H,T(p_0))\)
	\begin{multline*}
		p_0s\in C_{p_1r}=C_{p_1}\circ_{\gamma_{p_1}}C_r\iff p_2s=\gamma_{p_1}(p_0)s=\gamma_{p_1}(p_0s)\in C_r\\
		\iff q_2s\in C_{\nu_{p_2,q_2}(r)} \iff q_0s\in  C_{q_1\nu_{p_2,q_2}(r)}=C_{\nu_{p_0,q_0}(p_1r)} 
	\end{multline*}
	additionally, we also have	
		\[\gamma_{p}(p_0s)=\gamma_{p_1r}(p_0s)=\gamma_r(\gamma_{p_1}(p_0s))=\gamma_r(p_2s)=\gamma_{\nu_{p_2,q_2}(r)}(q_2s)=\gamma_{\nu_{p_2,q_2}(r)}(\gamma_{q_1}(q_0s))=\gamma_{\nu_{p_0,q_0}(p)}(p_1r).\]
	This shows that \(D_{p_0},D_{q_0},\nu_{p_0,q_0}\) satisfy the conditions of \defref{defn-enable}.

	For the ''only if'' implication we take \(D_{p_0},D_{q_0},\nu_{p_0,q_0}\) as in \defref{defn-enable}. From \rmkref{exp-enable-rmk}, we see that we may expand \(D_{p_0}\) and \(D_{p_0}\). Thus, we may assume that \(D_{p_0}\not\preceq p_1\) and \(D_{q_0}\not\preceq q_1\). We may then also assume that \(\forall p\in D_{p_0},\ C_p\neq \emptyset\) and \(\forall q\in D_{q_0},\ C_q\neq \emptyset\). Now to show that \(p_1\preceq D_{p_0}\) and \(q_1\preceq D_{q_0}\), we note that if some \(p\in D_{p_0}\) satisfies \(p_1\not\preceq p\) we must then have \(p_1\perp p\), since by assumption \(p\not\preceq p_1\), and thus \(C_{p_1}\perp C_p\) which contradicts 
		\[C_{p_1}\preceq p_0\preceq \bigcup_{r\in D_{p_0}}C_r\preceq C_p\] 
	since \(C_p\neq \emptyset\). Analogously, we can show that \(q_1\preceq D_{q_0}\).


	After  this we will proceed analogously as above by showing that the sets
		\begin{align*}
			D_{p_2}:&=\{r\in\mathcal{P}(G,T(p_1))\mid p_1r\in D_{p_0} \}\\
			D_{q_2}:&=\{r\in\mathcal{P}(G,T(q_1))\mid q_1r\in D_{q_0} \}
		\end{align*}
	together with the function \(\nu_{p_2,q_2}:D_{p_2}\to D_{q_2}\), uniquely determined by \(q_1\nu_{p_2,q_2}(r)=\nu_{p_0,q_0}(p_1r)\), satisfy the conditions from \defref{defn-enable}. We note that since 
		\[\gamma_{p_1}^{-1}(\bigcup_{r\in D_{p_2}} C_r)=C_{p_1}\circ_{\gamma_{p_1}}\bigcup_{r\in D_{p_2}} C_r=\bigcup_{r\in D_{p_2}} C_{p_1r}=\bigcup_{p\in D_{p_0}} C_p\]
	covers \(p_0\) exactly, \(\bigcup_{r\in D_{p_2}} C_r\) covers \(\gamma_{p_1}(p_0)=p_2\) exactly. Analogously we can show that \(\bigcup_{r\in D_{q_2}} C_r\) covers \(\gamma_{q_1}(q_0)=q_2\) exactly. For any \(r\in D_{p_2}\), and any \(s\in\mathcal{P}(H,T(p_2))\) we have	
		\[p_2s=\gamma_{p_1}(p_0s)\in C_{r}\iff p_0s\in C_{p_1r}\iff q_0s\in C_{\nu_{p_0,q_0}(p_1r)}\iff q_2s\in C_{\nu_{p_2,q_2}(r)}\]
	and
		\[\gamma_{r}(p_2s)=\gamma_{r}(\gamma_{p_1}(p_0s))=\gamma_{p_1r}(p_0s)=\gamma_{\nu_{p_0,q_0}(p_1r)}(q_0s)=\gamma_{\nu_{p_2,q_2}(r)}(q_2s).\]
	Showing that the conditions from \defref{defn-enable} are satisfied. 


\end{proof}

Using these results we can formulate a condition when an isomorphism described by a gluing diagram, is shift-surjective.
\begin{theorem}\label{shift-surj-thm}
	For any surjective gluing diagram \(\mathfrak{G}\), connecting \((G,R)\) to \((H,S)\) the path-monoid homomorphism \(\Gamma\) it describes is shift surjective if and only if for any \(v,w\in VG\), any \(p\in\mathcal{P}(H,x_v)\) and \(q\in\mathcal{P}(H,x_w)\) with \(T(p)=T(q)\) that are internal in \(B_v\) and \(B_w\) resp.,  \(\mathfrak{G}\) enables the shift from \(p\) to \(q\).
\end{theorem}
\begin{proof}
	We will first show that \(\mathfrak{G}\) enables the shift between any \(p_0,q_0\in\mathcal{P}(H,S)\) with \(u:=T(p_0)=T(q_0)\) not a sink. To do this we will take some \(e\in o^{-1}(u)\) and use that since \(\mathfrak{G}\) is surjective there exist some \(p_1\in \mathcal{P}(G,R)\) s.t. \(p_0\precneq p_0e\preceq C_{p_1}\) and \(C_{p_1}\neq \emptyset\). This means that for any \(p\in\mathcal{P}(G,R)\) with \(C_p\preceq p_0\) we must have \(p\precneq p_1\). We can take \(p_2\in\mathcal{P}(G,R)\) with maximal length s.t. \(C_{p_2}\preceq p_0\) and then we have \(p_0\not\preceq C_{p_2d}\) for each \(d\in o ^{-1}(T(p_2))\) and thus 
		\[\gamma_{p_2}(p_0)\not\preceq \bigcup_{d\in o ^{-1}(T(p)) } C_d=B_{T(p_2)},\]
	making it internal. Analogously, we can get \(q_2\in\mathcal{P}(G,R)\) with \(C_{q_2}\preceq q_0\) and \(\gamma_{q_2}(q_0)\) is internal in \(B_{T(p_2)}\). Now since by assumption the shift from \(\gamma_{p_2}(p_0)\) to \(\gamma_{q_2}(q_0)\) is enabled by \(\mathfrak{G}\). Thus, by \lemref{enable-image-lem} and \lemref{shift-translation-lem}, \(\pi_{p_0,q_0}\) is in the image of \(\Gamma\).

	Now we will just have to deal with \(p_0,q_0\in\mathcal{P}(H,S)\) with \(u:=T(p_0)=T(q_0)\) being a sink. In this case since \(p_0\), \(q_0\) are each prefixes only of themselves and \(\mathfrak{G}\) is surjective we must have some \(p_1,q_1\in\mathcal{P}(G,R)\) s.t. \(C_{p_1}=\{p_0\}\) and \(C_{q_1}=\{q_0\}\). This means that we have \(\Gamma(\pi_{p_1,q_1})=\pi_{p_0,q_0}\).
	
	The converse follows from \lemref{enable-image-lem} and \lemref{shift-translation-lem}.
\end{proof}

While gluing diagrams are a convenient way of describing shift-preserving homomorphisms, they do not describe them uniquely, as two different diagrams can easily describe the same function. While showing exactly when this is the case, is not necessary for the purposes of this article we will show one way of modifying gluing diagrams that doesn't change the described homomorphism. For this we will first look at a way of connecting \(\mathcal{P}(G,x)\) and \(\mathcal{P}(G,x^{\mathbf{v}})\) for any \(\mathbf{v}\in x\). We remind of the convention that \(\mathbf{M}_{x^{\mathbf{v}}}=\mathbf{M}_x\setminus\{\mathbf{v}\}\sqcup o^{-1}(v)\), with the underlying object of any \(e\in o^{-1}(v)\) being \(t(e)\).

\begin{defn}
	For any \(x\in\mathcal{FM}_G\), \(\mathbf{v}\in x\), we define the function
		\[\kappa_{\mathbf{v}}: \mathcal{P}(G,x^{\mathbf{v}})\to\mathcal{P}(G,x)\setminus \{\varepsilon_{\mathbf{v}}\}\]
	as follows for any  \(p\in \mathcal{P}(G,x^{\mathbf{v}})\)
		\[\kappa_{\mathbf{v}}(p)=	\begin{cases}
										e_{\mathbf{v}}p, & O(p)=e\in o^{-1}(v)\subseteq\mathbf{M}_{x^{\mathbf{v}}}\\
										p, &  O(p)\not\in o^{-1}(v)
									\end{cases},\] 
	where \(e_{\mathbf{v}}\) is the element in \(\mathcal{P}(G,x)\) with underlying path \(e\in o^{-1}(v)\subseteq \mathcal{P}(G)\) and \(O(e_{\mathbf{v}})=\mathbf{v}\).


\end{defn}
We note that \(\kappa_{\mathbf{v}}\) is a bijection since any non-empty path \(p\in \mathcal{P}(G,x)\) with \(O(p)=\mathbf{v}\) is of the form  \(e_{\mathbf{v}}q\) for some \(e\in o^{-1}(v)\) and some \(q\in\mathcal{P}(G,T(q))\). The key property of this function is the following.
\begin{lemma}\label{cutting-gluing-lem}
	For any independent \(C\subseteq\mathcal{P}(G,R)\), \(p\in C\), \(x\in \mathcal{FM}_G\) and any terminus-maintaining, injective function \(\gamma:C\to x\), which we will expand to \(\gamma:\mathcal{P}(G,R)_C\to \mathcal{P}(G,x)\) as usual, we have for any \(p\in C\)
		\[\kappa_{\gamma(p)}\circ\gamma^p=\gamma|_{\mathcal{P}(G,R)_{C^p}}\] 	
	This means especially that	
		\[C^p\circ_{\gamma^p}\kappa^{-1}_{\gamma(p)}(D)=C\circ_{\gamma} D\]
	for any  \(D\subseteq \mathcal{P}(G,x)\setminus \{\varepsilon_{\mathbf{v}}\}\).
\end{lemma}
\begin{proof}
	Firstly we not that the equation has to be shown only for paths in \(C^p\) since any \(q\in \mathcal{P}(G,R)_{C^p}\) can be written as \(\tilde{q}r\) where \(\tilde{q}\in C^p\) and \(r\in \mathcal{P}(G,T(\tilde{p})))\) and so 
		\[\kappa_{\gamma(p)}(\gamma^p(\tilde{q}r))=\kappa_{\gamma(p)}(\gamma^p(\tilde{q}))r \text{ and } \gamma(\tilde{q}r)=\gamma(\tilde{q})r.\]

	For paths \(q\in C^p\) we will distinguish the cases where \(p\not\preceq q\) and \(p\preceq q\).
	In the case \(p\not\preceq q\) we also have \(q\in C\) and \(\gamma^p(q)\notin	o^{-1}(e)\) and thus
		\[\kappa_{\gamma(p)}(\gamma^p(q))=\gamma^p(q)=\gamma(q).\]
	On the other hand if we have \(p\preceq q\) we  have some \(e\in o^{-1}(v)\) s.t \(q=pe\) and thus	
		\[\kappa_{\gamma(p)}(\gamma^p(q))=\kappa_{\gamma(p)}(\gamma^p(pe))=\kappa_{\gamma(p)}(\varepsilon_{e})=e_{\gamma(p)}=\gamma(pe),\]
	here we again assume that \(o^{-1}(v)\subseteq \mathbf{M}_{x^{\gamma(p)}}\).

	The second point follows from the first one since by definition we have
		\[C^p\circ_{\gamma^p}\kappa^{-1}_{\gamma(p)}(D)=(\gamma^p)^{-1}(\kappa^{-1}_{\gamma(p)}(D))=\gamma^{-1}(D)=C\circ_{\gamma} D.\]


\end{proof}

We will now define how one can expand a gluing diagram along an unblocked element.
\begin{defn}\label{exp-defn}
	For any gluing diagram \(\mathfrak{G}=(x_v,C_{\varepsilon_R},\gamma_{\varepsilon_R},C_e,\gamma_e)_{v\in VG,e\in EG}\), any unblocked \(\mathbf{u}\in x_w\) for some \(w\in VG\), we can define \(\mathfrak{G}^{\mathbf{u}}=(C^{\mathbf{u}}_{\varepsilon_R},\gamma^{\mathbf{u}}_{\varepsilon_R},y_v,C^{\mathbf{u}}_e,\gamma^{\mathbf{u}}_e)_{v\in VG,e\in EG}\) as follows:
	\begin{itemize}
		\item \(y_v:=	\begin{cases}
							x_v, & v\neq w\\
							x^{\mathbf{u}}_w, & v=w
						\end{cases}\).
		\item \(C^{\mathbf{u}}_{\varepsilon_R}:=\begin{cases}
													C^{\gamma ^{-1}_{\varepsilon_R}(\mathbf{u})}_{\varepsilon_R},& R=w\\
													C_{\varepsilon_R}, & R\neq w
												\end{cases}\) and \(\gamma^{\mathbf{u}}_{\varepsilon_R}:=\begin{cases}
													\gamma^{\gamma ^{-1}_{\varepsilon_R}(\mathbf{u})},& R=w\\
													\gamma_{\varepsilon_R}, & R\neq w
												\end{cases}\)
		\item \(C^{\mathbf{u}}_{e}:=	\begin{cases}
											C_{e}, & o(e)\neq w,t(e)\neq w\\
											C^{\gamma ^{-1}_{e}(\mathbf{u})}_{e},& o(e)\neq w,t(e)= w\\
											\kappa^{-1}_{\mathbf{u}}(C_e), & o(e)=w, t(e)\neq w\\
											\kappa^{-1}_{\mathbf{u}}(C^{\gamma ^{-1}_e(\mathbf{u})}_e), & o(e)=w, t(e)=w	
										\end{cases}\) 
		and \(\gamma^{\mathbf{u}}_{e}:=	\begin{cases}
												\gamma_{e}, & o(e)\neq w,t(e)\neq w\\
												\gamma^{\gamma ^{-1}_{e}(\mathbf{u})}_{e},& o(e)\neq w,t(e)= w\\
												\gamma_e\circ \kappa_{\mathbf{u}}, & o(e)=w, t(e)\neq w\\
												\gamma^{\gamma ^{-1}_e(\mathbf{u})}_e\circ\kappa_{\mathbf{u}}, & o(e)=w, t(e)=w	
											\end{cases}\) 
								
	\end{itemize}
	For a floating diagram \(\mathfrak{G}=(x_v,C_e,\gamma_e)_{v\in VG,e\in EG}\) we will define \(\mathfrak{G}^{\mathbf{u}}=(y_v,C^{\mathbf{u}}_e,\gamma^{\mathbf{u}}_e)_{v\in VG,e\in EG}\) as above.
	\end{defn}
	To see how expanding changes the visual representation of the gluing diagram from \figref{fig1}, see \figref{fig2}.
	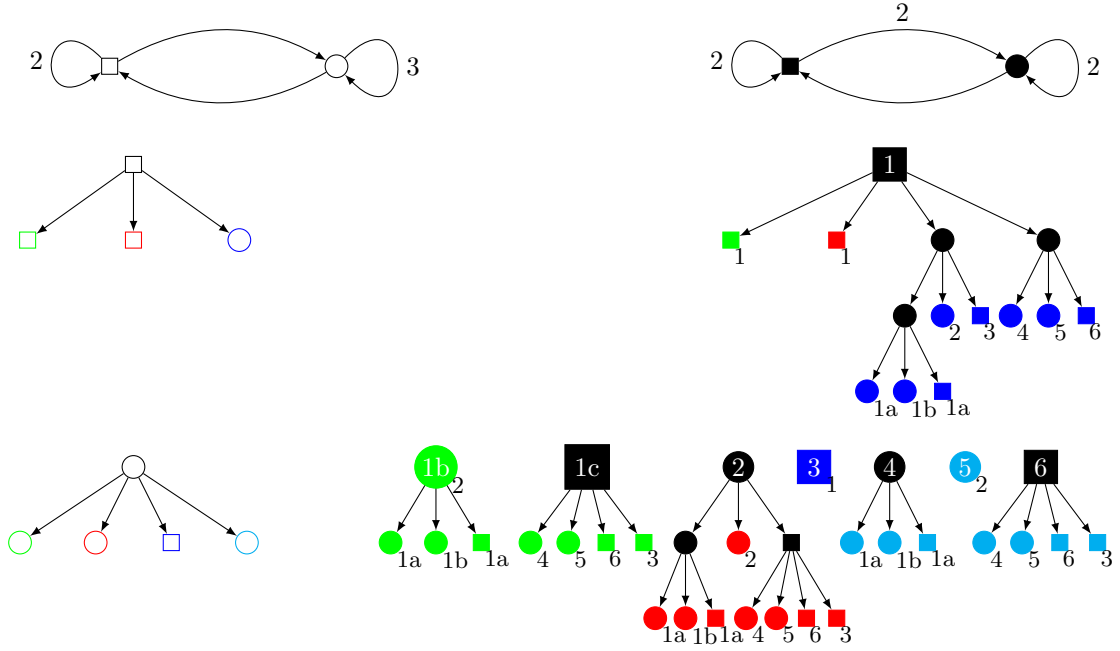
\begin{figure}
		\begin{center}
			\begin{tikzpicture}
				\node[big hollow square]     (x) at (-6,0)   {};
				\node[big hollow circle]   (y) at (-3,0)    {};

				\path[-latex]
				(x)     edge[bend left]                         node[above]             {} (y)
				edge[out=135, in=215, looseness=15]     node[left]              {2} (x)
				(y)     edge[bend left]                         node[below]             {}  (x)
				edge[out=45, in=315, looseness=15]      node[right]             {3} (y)
				;
				\node[big solid square]     (z) at (3,0)   {};
				\node[big solid circle]   (w) at (6,0)    {};

				\path[-latex]
				(z)     edge[bend left]                         node[above]             {2} (w)
				edge[out=135, in=215, looseness=15]     node[left]              {2} (z)
				(w)     edge[bend left]                         node[below]             {}  (z)
				edge[out=45, in=315, looseness=15]      node[right]            {2} (w)
				;
			\end{tikzpicture}
			\begin{tikzpicture}
				[edge from parent/.style={draw,-latex},
					level distance=10mm,
					level 1/.style={sibling distance=14mm},
					level 2/.style={sibling distance=5mm,},
					level 3/.style={sibling distance=3mm,},
					label distance=-2mm]
				\node[big hollow square] (r1) at (-5,0) {}
				child {node[big hollow green square] {}}
				child {node[big hollow red square] {}}
				child {node[big hollow blue circle] {}};
				\node[big solid square] (r2) at (5,0) {\textcolor{white}{1}}
				child {node[big solid green square, label={[font=\small,text=black]-10:$1$}] {}}
				child {node[big solid red square, label={[font=\small,text=black]-10:$1$}] {}}
				child {node[big solid circle] {}
						child {node[big solid circle] {}
							child[sibling distance=5mm] {node[big solid blue circle, label={[font=\small,text=black]-10:$1$a}] {}}
							child[sibling distance=5mm] {node[big solid blue circle, label={[font=\small,text=black]-10:$1$b}] {}}
							child[sibling distance=5mm] {node[big solid blue square, label={[font=\small,text=black]-10:$1$a}] {}}}
						child {node[big solid blue circle, label={[font=\small,text=black]-10:$2$}] {}}
						child {node[big solid blue square, label={[font=\small,text=black]-10:$3$}] {}}
					}
				child {node[big solid circle] {}
						child {node[big solid blue circle, label={[font=\small,text=black]-10:$4$}] {}}
						child {node[big solid blue circle, label={[font=\small,text=black]-10:$5$}] {}}
						child {node[big solid blue square, label={[font=\small,text=black]-10:$6$}] {}}
					};
				\node[big hollow circle] (r1) at (-5,-4) {}
				child[sibling distance=10mm] {node[big hollow green circle] {}}
				child[sibling distance=10mm] {node[big hollow red circle] {}}
				child[sibling distance=10mm] {node[big hollow blue square] {}}
				child[sibling distance=10mm] {node[big hollow cyan circle] {}};

				\node[big solid circle] (r21) at (-1,-4)	{\textcolor{white}{1a}}
					child[sibling distance=6mm] {node[big solid green circle, label={[font=\small,text=black]-10:$1$a}] {}}
					child[sibling distance=6mm] {node[big solid green circle, label={[font=\small,text=black]-10:$1$b}] {}}
					child[sibling distance=6mm] {node[big solid green square, label={[font=\small,text=black]-10:$1$a}] {}};
				\node[big solid green circle, label={[font=\small,text=black]-10:$2$} ] (r22) at (0,-4)		{\textcolor{white}{1b}};
				\node[big solid square] (r23) at (1,-4) 	{\textcolor{white}{1c}}
						child[sibling distance=5mm] {node[big solid green circle, label={[font=\small,text=black]-10:$4$}] {}}
						child[sibling distance=5mm] {node[big solid green circle, label={[font=\small,text=black]-10:$5$}] {}}
						child[sibling distance=5mm] {node[big solid green square, label={[font=\small,text=black]-10:$6$}] {}}
						child[sibling distance=5mm] {node[big solid green square, label={[font=\small,text=black]-10:$3$}] {}};

				\node[big solid circle] (r3) at (3,-4) {\textcolor{white}{2} }
				child[sibling distance=7mm] {node[big solid circle] {}
					child[sibling distance=4mm] {node[big solid red circle, label={[font=\small,text=black]-10:$1$a}] {}}
					child[sibling distance=4mm] {node[big solid red circle, label={[font=\small,text=black]-10:$1$b}] {}}
					child[sibling distance=4mm] {node[big solid red square, label={[font=\small,text=black]-10:$1$a}] {}}}
				child[sibling distance=7mm] {node[big solid red circle, label={[font=\small,text=black]-10:$2$}] {}}
				child[sibling distance=7mm] {node[big solid square] {}
						child[sibling distance=4mm] {node[big solid red circle, label={[font=\small,text=black]-10:$4$}] {}}
						child[sibling distance=4mm] {node[big solid red circle, label={[font=\small,text=black]-10:$5$}] {}}
						child[sibling distance=4mm] {node[big solid red square, label={[font=\small,text=black]-10:$6$}] {}}
						child[sibling distance=4mm] {node[big solid red square, label={[font=\small,text=black]-10:$3$}] {}}
					};

				\node[big solid blue square, label={[font=\small,text=black]-10:$1$}] (r4) at (4,-4) {\textcolor{white}{3}};

				\node[big solid circle] (r4) at (5,-4) {\textcolor{white}{4} }
					child[sibling distance=5mm] {node[big solid cyan circle, label={[font=\small,text=black]-10:$1$a}] {}}
					child[sibling distance=5mm] {node[big solid cyan circle, label={[font=\small,text=black]-10:$1$b}] {}}
					child[sibling distance=5mm] {node[big solid cyan square, label={[font=\small,text=black]-10:$1$a}] {}};

				\node[big solid cyan circle, label={[font=\small,text=black]-10:$2$}] (r5) at (6,-4) {\textcolor{white}{5} };

				\node[big solid square] (r5) at (7,-4) {\textcolor{white}{6}}
				child[sibling distance=5mm] {node[big solid cyan circle, label={[font=\small,text=black]-10:$4$}] {}}
				child[sibling distance=5mm] {node[big solid cyan circle, label={[font=\small,text=black]-10:$5$}] {}}
				child[sibling distance=5mm] {node[big solid cyan square, label={[font=\small,text=black]-10:$6$}] {}}
				child[sibling distance=5mm] {node[big solid cyan square, label={[font=\small,text=black]-10:$3$}] {}};
			\end{tikzpicture}
		\end{center}
		\caption{ The floating gluing diagram from \figref{fig1} expanded by the solid circle numbered \(1\) in \(x_{\circ}\)}
		\label{fig2}
	\end{figure}
	
	We can show that expanding a gluing diagram does not change the path monoid function that it describes, using \lemref{cutting-gluing-lem}. We will also require the observation that 
		\[C\circ_{\gamma}D^r=\gamma^{-1}(D^r)=\gamma^{-1}(D)^{\gamma^{-1}(r)}=(C\circ_{\gamma} D)^{\gamma^{-1}(r)}\]
	for any \(C\subseteq \mathcal{P}(G,R)\), \(x\in\mathcal{FM}_H\), any bijection \(\gamma:C\to x\) and any \(D\subseteq \mathcal{P}(D,x)\).
	\begin{prop}\label{exp-no-chg-prop}
		For any gluing diagram \(\mathfrak{G}=(x_v,C_{\varepsilon_R},\gamma_{\varepsilon_R},C_e,\gamma_e)_{v\in VG,e\in EG}\), any \(w\in VG\) and any unblocked \(\mathbf{u}\in x_w\) the same path monoid homomorphism described by \(\mathfrak{G}\) and \(\mathfrak{G}^{\mathbf{u}}\)
	\end{prop}
	\begin{proof}
		Using the definition of the path monoid homomorphism described by a gluing diagram, this is equivalent to showing 
			\[\sum_{c\in C_p}c=\sum_{d\in C^{\mathbf{u}}_p}d\]
		for any \(p\in\mathcal{P}(G,R)\). To do this we will show via induction that
			\[C_p^{\mathbf{u}}=	\begin{cases}
									C_p, & T(p)\neq u\\
									C^{\gamma^{-1}_p(\mathbf{u})}_p, & T(p)=u
								\end{cases}\]
		and 
			\[\gamma^{\mathbf{u}}_p=\begin{cases}
									\gamma_p, & T(p)\neq w\\
									\gamma^{\gamma^{-1}_p(\mathbf{u})}_p, & T(p)=w
								\end{cases}.\]
		The beginning case for \(p=\varepsilon_R\) follows from the definition of expanding a gluing diagram. For the induction step we take some \(p\in \mathcal{P}(G,R)\) and any \(e\in o^{-1}(T(p))\), we will consider four cases:
		\begin{itemize}
			\item If \(T(p)=o(e)\neq w\) and \(t(e)\neq w\) we simply have 
					\[C_{pe}^{\mathbf{u}}=C^{\mathbf{u}}_p\circ_{\gamma^{\mathbf{u}}_p}C^{\mathbf{u}}_e=C_p\circ_{\gamma_p}C_e=C_{pe}\]
				and since \(\gamma^{\mathbf{u}}_e=\gamma_e\) we also get \(\gamma^{\mathbf{u}}_{pe}=\gamma_{pe}\)
			\item If \(T(p)=o(e)= w\) and \(t(e)\neq w\) we have
					\[C_{pe}^{\mathbf{u}}=C^{\mathbf{u}}_p\circ_{\gamma^{\mathbf{u}}_p}C^{\mathbf{u}}_e=C^{\gamma^{-1}_p(\mathbf{u})}_p\circ_{\gamma^{\gamma^{-1}_p(\mathbf{u})}_p}\kappa^{-1}_{\mathbf{u}}(C_e)=C_p\circ_{\gamma_p} C_e=C_{pe}\]
				and  
				\[\gamma^{\mathbf{u}}_{pe}=\gamma^{\mathbf{u}}_e\circ \gamma^{\mathbf{u}}_p=\gamma_e\circ\kappa_{\mathbf{u}}\circ\gamma^{\gamma^{-1}_p(\mathbf{u})}_p=\gamma_e\circ \gamma_p|_{\mathcal{P}(G,R)_{C_p^p}}=\gamma_{pe},\]
				using \lemref{cutting-gluing-lem}.
				
			\item If \(T(p)=o(e)\neq w\) and \(t(e)= w\) we have
					\[C_{pe}^{\mathbf{u}}=C^{\mathbf{u}}_p\circ_{\gamma^{\mathbf{u}}_p}C^{\mathbf{u}}_e=C_p\circ_{\gamma_p} C^{\gamma ^{-1}_{e}(\mathbf{u})}_e=C_{pe}^{\gamma ^{-1}_{pe}(\mathbf{u})}\]
				and since \(\gamma^{\mathbf{u}}_{e}=\gamma^{\gamma ^{-1}_e(\mathbf{u})}_e\) we also have \(\gamma_{pe}=\gamma^{\gamma^{-1}_{pe}(\mathbf{u})}_{pe}\)
			\item If  \(T(p)=o(e)= w\) and \(t(e)= w\) we have
					\[C_{pe}^{\mathbf{u}}=C^{\mathbf{u}}_p\circ_{\gamma^{\mathbf{u}}_p}C^{\mathbf{u}}_e=C^{\gamma^{-1}_p(\mathbf{u})}_p\circ_{\gamma^{\gamma^{-1}_p(\mathbf{u})}_p}\kappa^{-1}_{\mathbf{u}}(C^{\gamma ^{-1}_e(\mathbf{u})}_e)=C_p\circ_{\gamma_p} C^{\gamma ^{-1}_e(\mathbf{u})}_e=C^{\gamma ^{-1}_{pe}(\mathbf{u})}_{pe}\]
				and since \(\gamma^{\mathbf{u}}_e=\gamma^{\gamma ^{-1}_e(\mathbf{u})}_e\circ\kappa_{\mathbf{u}}\) we also have \(\gamma^{\mathbf{u}}_{pe}=\gamma^{\gamma^{-1}_{pe}(\mathbf{u})}_{pe}\), the same way as in the second point using \lemref{cutting-gluing-lem}.	
		\end{itemize}
	\end{proof}
	

We note that this definition only makes sense for \(\mathbf{u}\) unblocked because of the image of \(\kappa_{\mathbf{u}}\) does not include \(\varepsilon_{\mathbf{u}}\).

\section{Example of application}\label{H-T-sec}

In this section we will use the tools from the previous sections to classify the pseudogroups of graph shift that consists of one source and one other vertex. This will provide a solution to the isomorphism problem of Higman-Thompson groups. The first full solution to this problem was provided in \cite{Pardo2011}, with a necessary condition already provided in \cite{higman74}. Our solution will be very elementary and provide a clear method for constructing these isomorphisms.

Firstly we will define the full group of the shift pseudogroup group.
\begin{defn}
	A shift \(\pi:\mathcal{M}_p(G,R)_x \to \mathcal{M}_p(G,R)_y\) is called \textbf{full} if \(x=y=R\) in the path monoid.  The group formed by full shifts under composition will be called the \textbf{full group of shifts} of \((G,R)\) and denoted by \(\mathcal{FS}(G,R)\).
\end{defn}
Since compositions and inverses of full shifts are also full shifts and any full shift can be composed with any other full shift, the full group of shifts is, in fact, a group. Note that if we have a root- and shift-preserving homomorphism \(\Gamma:\mathcal{M}_p(G,R)\to \mathcal{M}_p(H,S)\), we can restrict an induced shift pseudogroup homomorphism \(\Gamma_s:\mathcal{S}(G,R)\to \mathcal{S}(H,S)\) to a homomorphism between the full groups \(\Gamma_f=\Gamma_s|_{\mathcal{FS}(G,R)}\). If \(\Gamma_s\) is an isomorphism so is \(\Gamma_f\).  

Firstly we will define the collection of graphs that we will be dealing with.
\begin{defn}
	 For any \(a\in \mathbb{Z}_{\geq 1}\) and \(n\in \mathbb{Z}_{\geq 2}\) the graphs \(G_{a,n}\) and \(G_n\) by setting 
	\begin{itemize}
		\item \(VG_n=\{v\}\) and \(VG_{a,n}=\{R,v\}\)
		\item \(EG_n=\{e_0,e_1,\dots,e_{n-1}\}\) and \(EG_{a,n}=\{d_0,d_1,\dots d_{a-1}\}\cup EG_n\)
		\item \(\forall 0\leq i<n,\ o_{G_n}(e_i)=o_{G_{a,n}}(e_i)=t_{G_n}(e_i)=t_{G_{a,n}}(e_i)=v\)
		\item \(\forall 0\leq j<a,\ o_{G_{a,n}}(d_j)=R,\ t_{G_{a,n}}(d_j)=v\).
	\end{itemize}
\end{defn}
\begin{figure}
	\begin{center}
		\begin{tikzpicture}
			\node[big hollow square]   	(x) at (-4,0)   {};
			\node[big hollow circle]   	(y) at (-2,0)  	{};

			\path[-latex]
			(x)     edge[]                         				node[above]             {a} (y)
			(y)     edge[out=45, in=315, looseness=15]      	node[right]             {n} (y)
			;
			\node[] (a) at (-3,-1) {\(G_{a,n}\)}; 
			\node[big solid circle]     (z) at (3,0)   {};

			\path[-latex]
			(z)     edge[out=45, in=315, looseness=15]     		node[right]              {n} (z)
			;
			\node[] (a) at (3,-1) {\(G_{n}\)};
		\end{tikzpicture}
	\end{center}
	\caption{The graphs \(G_{a,n}\) and \(G_n\)}
\end{figure}
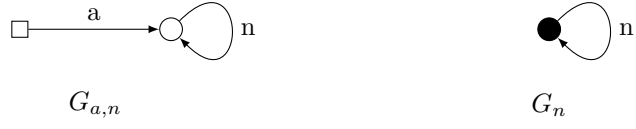
We note that \(G_n\) is a subgraph of \(G_{a,n}\), \(G_n\) is a subgraph of \(G_m\) and \(G_{a,n}\) is a subgraph of \(G_{b,m}\), whenever \(a\leq b\) and \(n\leq m\). Following the framework from \cite{scott84} we can define the Higman-Thompson group.

\begin{defn}
	The Higman-Thompson groups \(V_{a,n}\) are defined as 
		\[V_{a,n}:=\mathcal{FS}(G_{a,n},R)\]
\end{defn}

This means that any shift-preserving homomorphism between \(\mathcal{M}_p(G_{a,n},R)\) and \(\mathcal{M}_p(G_{b,m},R)\) induces homomorphisms between \(V_{a,n}\) and \(V_{b,m}\). To get a clearer picture of such homomorphisms we can give concrete conditions when there is a gluing diagram connecting \(G_{a,n}\) and \(G_{b,m}\)

\begin{prop}
	For any \(a,b\geq 1\), \(n,m\geq 2\) there exist a gluing diagram connecting \((G_{a,n},R)\) and \((G_{b,m},R)\) if and only if there exists some \(l\geq 1\) s.t.:
	\begin{itemize}
		\item \(\exists k_1\geq 0,\ l\cdot a= b+k_1\cdot(m-1)\)
		\item \(\exists k_2\geq 0,\ l\cdot n =l+k_2\cdot (m-1)\).
	\end{itemize}
\end{prop}
\begin{proof}
	Firstly, we note that for any \(l\geq 0\) all the bases of \(\mathcal{P}(G_{b,m},l\cdot v)\) have cardinality \(l+k\cdot(m-1)\) for some \(k\geq 0\) and for any \(k\geq 0\) there exists a basis \(B\subseteq\mathcal{P}(G_{b,m},l\cdot v)\) with \(|B|=l+k\cdot(m-1)\). Additionally, any basis of \(\mathcal{P}(G_{b,m},R)\) has cardinality either \(1\) or \(b+k\cdot (m-1)\) for some \(k\) and there exist bases with all these cardinalities. This can be seen by a straightforward inductive argument using expansions.

	For the ``if'' direction of the implication we take bases \(B_R\subseteq \mathcal{P}(G_{b,m},R),B_v\subseteq\mathcal{P}(G_{b,m},l\cdot v)\) with 
		\[|B_R|=b+k_1\cdot(m-1)=l\cdot a \text{ and }|B_v|=l+k_2\cdot(m-1)=l\cdot n\]
	This allows us to partition the bases
		\[B_R=\bigsqcup_{1\leq j \leq a} C_{d_j} \text{ and }B_v=\bigsqcup_{1\leq i \leq n} C_{e_i}\]
	s.t. \(|C_{d_j}|=l\) and \(|C_{e_i}|=l\). We can also assume that \(B_R\neq \{\varepsilon_R\}\) i.e. \(R\notin T(B_R)\), by expanding it. This allows us to take arbitrary terminus preserving bijections \(\gamma_{d_j}:C_{d_j}\to l\cdot v\), for each \(1\leq j \leq a\), and \(\gamma_{e_i}:C_{e_i}\to l\cdot v\), for each \(1\leq i\leq n\). Now we only have to set \(x_v:=l\cdot v\), \(x_R:=R\), \(C_{\varepsilon_R}:=\{\varepsilon_R\}\) and \(\gamma_{\varepsilon_R}:\{\varepsilon_R\}\to R\) to obtain a gluing diagram between \((G_{a,n},R)\) and \((G_{b,m},R)\).

	For the ''only if'' direction we take a gluing diagram \(\mathfrak{G}\). We first may assume that no element of \(x_R,x_v\) has \(R\) as their underlying element, since if we expand \(\mathfrak{G}^\mathbf{R}\) by some element with underlying vertex \(R\) we reduce the number of such elements in \(x_R,x_v\). This means that we can write \(x_R=k\cdot v\) and \(x_v=l\cdot v\), for some \(k,l\geq 1\). Since \(C_{\varepsilon_R}\subseteq\mathcal{P}(G_{b,m},R)\) is a basis(not containing \(\varepsilon_R\) by assumption) we have some \(k_0\geq 0\) s.t. 
		\[k=|C_{\varepsilon_R}|=b+k_0\cdot(m-1).\]
	Since \(B_R\subseteq \mathcal{P}(G,k\cdot v)\) is a basis we get some \(\tilde{k}_1\geq 0\) s.t.
		\[l\cdot a=\sum_{j=0}^{a-1}|C_{d_j}|=|B_R|=k+\tilde{k}_1\cdot (m-1)=b+(k_0+\tilde{k}_1)\cdot(m-1).\]
	And since \(B_v\subseteq  \mathcal{P}(G,l\cdot v)\) is a basis, as well, we have some \(k_2\geq 0\) s.t. 
		\[l\cdot n = \sum_{i=0}^{n-1}|C_{e_i}|=|B_v|=l+k_2\cdot(m-1).\]
	By taking \(k_1=k_0+\tilde{k}_1\) we see that the properties of the lemma are satisfied.
\end{proof}
We note that in any gluing diagram connecting \((G_{a,n},R)\) and \((G_{b,m},R)\) we have \(x_R\neq 0,x_v\neq 0\) since \(C_{\varepsilon_R}\neq \emptyset\) and \(C_{d_j}\neq \emptyset\) for some (or equivalently any) \(1\leq j< a\). This means that the gluing diagrams from the above lemma induce embeddings \(V_{a,n}\hookrightarrow V_{b,m}\). This can be contrasted with the results of \cite[Corollary 11.16]{mattebon2018} and \cite{Birget02082020}, where the embeddings \(V_{1,n}\hookrightarrow V_{1,m}\) where proven to exist and described resp., for any \(m,n\geq 2\). Therefore, not all embeddings between Higman-Thompson groups can be constructed via gluing diagrams. On the other hand, the following will show that whenever there exists an isomorphism between Higman-Thompson groups, one can be described via gluing diagrams. 

We will first show a necessary condition for there being a root- and terminus-preserving isomorphism between \(\mathcal{M}_p(G_{a,n},R)\) and \(\mathcal{M}_p(G_{b,m})\).
\begin{lemma}\label{first-impl-lem}
	For any \(a,b\in\mathbb{Z}_{\geq 1}\) and \(n,m\in\mathbb{Z}_{\geq 2}\) if there exists a root- and terminus-preserving isomorphism
		\[\Gamma:\mathcal{M}_p(G_{a,n},R)\to \mathcal{M}_p(G_{b,n},R),\]
	s.t. \(\Gamma^{-1}\) is also terminus preserving, then
		\[n=m \text{ and }\gcd(a,n-1)=\gcd(b,m-1).\]
	If \(\Gamma\) is an isomorphism we have in fact \(m=n\)
\end{lemma}
\begin{proof}
	We can factor \(\Gamma\) over the terminus function to get a homomorphism \(\Gamma_T:\mathcal{M}_{G_{a,n}}\to \mathcal{M}_{G_{b,m}}\) s.t. 
		\[\forall x\in \mathcal{M}_p(G_{a,n},R),\ T_{\mathcal{M}}(\Gamma(x))=\Gamma_T(T_{\mathcal{M}}(x)).\]
	As \(\Gamma^{-1}\) is also terminus preserving, \(\Gamma_T\) is in fact an isomorphism.
	We note that \(\mathcal{M}_{G_{a,n}}=\{0\}\cup\{v,2v,\dots, (n-1)\cdot v\}\) and \(\{v,2v,\dots, (n-1)\cdot v\}\) form the cyclic group of order \(n-1\) (with \((n-1)\cdot v\) acting as the neutral element). So since \(\Gamma_T(0)=0\) it restricts to an isomorphism between the cyclic groups of order \(n-1\) and \(m-1\), which means that we have to have \(n=m\). For the second condition we note that
		\[\Gamma_T(a\cdot v)=\Gamma_T(R)=T_{\mathcal{M}_{G_{a,n}}}(\Gamma(\varepsilon_R))=T_{\mathcal{M}_{G_{b,n}}}(\varepsilon_R)=R=b\cdot v\]
	which means that \(a\) has the same order in \(\mathbb{Z}/{(n-1)Z}\) as \(b\) in \(\mathbb{Z}/{(m-1)Z}\) and thus 
		\[\gcd(a,n-1)=\gcd(b,m-1).\]
\end{proof}

We note that if an isomorphism of path monoids that is shift-preserving and shift-surjective, has a shift-surjective (and thus terminus-preserving) inverse, and so we can apply the above lemma. 
This condition is analogous to the necessary condition for Higman-Thompson groups to be isomorphic shown in \cite[Theorem 6.4.]{higman74} (using Rubin's theorem we can show that any isomorphism of Higman-Thompson groups induces a shift-preserving and shift-surjective isomorphism of path monoids and so Higman's condition is implied by the above lemma). We will show that the condition is sufficient for having a shift-preserving isomorphism as well. The first step is to reduce the problem to finding a floating basis from \(G_n\) to itself, with \(x_v= l\cdot v\in \mathcal{FM}_{G_{n}}\) for any \(l\) that is coprime to \(n-1\).

\begin{lemma}\label{floating-to-rooted-lem}
	For any \(a,b\in\mathbb{Z}_{\geq 1}, n\in \mathbb{Z}_{\geq 2}\), s.t. there exists \(l,k\in\mathbb{Z}_{\geq 0}\), with \(l\geq 1\) and
		\[l\cdot a=k\cdot(n-1)+b,\]
	if there exists a floating gluing diagram \((B_v,x_v,C_e,\gamma_e)_{e\in EG_n}\) connecting \(G_n\) to itself, with \(x_v=l\cdot v\), there exists a gluing diagram connecting \((G_{a,n},R)\) and \((G_{b,n},R)\). If the former gluing diagram is shift surjective, then so is the latter.
\end{lemma}
\begin{proof}
	We note that there exists a basis \(B_R\) of \(\mathcal{P}(G_{b,n},R)\) with \(|B_R|=k\cdot(n-1)+b\), since \(B_0=\{d_0,\dots,d_{b-1}\}\subseteq\mathcal{P}(G_{b,n},R)\) is a basis with cardinality \(b\) and for any basis \(B_0\preceq B\subseteq \mathcal{P}(G_{c,n})\) we have \(|B^p|=|B|+(n-1)\) for any \(p\in B\), allowing us to construct \(B_R\).
	Since \(|B_R|=k\cdot(n-1)+b=l\cdot a\) we can partition it arbitrarily into \(a\) independent sets of cardinality \(l\)
		\[B_R=\bigsqcup_{0\leq j <a} C_{d_j}\]
	and endow them with (again arbitrary) functions \(\gamma_{d_j}:C_{d_j}\to x_v=l\cdot v\in\mathcal{FM}_{G_{a,n}}\), since any path in \(B_R\) must end in \(v\) (as \(B\neq \{\varepsilon_R\}\)). 
	
	Now since  \(B_v\subseteq \mathcal{P}(G_n,x_v)= \mathcal{P}(G_{b,n},x_v)\) is also a basis in \(\mathcal{P}(G_{b,n},x_v)\) we can construct the gluing diagram consisting of \((B_v,x_v,C_e,\gamma_e)_{e\in EG_n}\), \((B_R, R, C_d,\gamma_d)_{d\in EG_{b,n}\setminus EG_n}\) and \(C_{\varepsilon_R}=\{\varepsilon_R\}\subseteq\mathcal{P}(G_{b,n},R),\gamma_{\varepsilon_R}: \{\varepsilon_R\}\to R\) connecting \((G_{a,n},R)\) and \((G_{b,n},R)\).

	We note that \(B_R\) is already a splitting basis, so if \((B_v,x_v,C_e,\gamma_e)_{e\in EG_n}\) has an iterated splitting basis (i.e. is surjective), so does the gluing diagram constructed above.

	If \((B_v,x_v,C_e,\gamma_e)_{e\in EG_n}\) is shift surjective, then the isomorphism \(\Gamma:\mathcal{M}_p(G_{a,n},R)\to\mathcal{M}_p(G_{b,n},R)\), described by the non-floating gluing diagram constructed above, has every elementary shift \(\pi_{p,q}\) for any \(B_R\preceq p,q\in \mathcal{P}(G_{b,n},R)\) in its image due to the gluing diagram being unblocked and \lemref{shift-translation-lem}. For every \(p,q\in \mathcal{P}(G_{b,n},R)\) we can take an independent set \(B\subseteq\mathcal{P}(G_{b,n})\) of paths that covers \(p\) exactly and is long enough s.t. \(B_R\preceq B,\pi(B)\) then since
		\[\pi_{p,q}=\bigoplus_{r\in B} \pi_{r,\pi_{p,q}(r)},\]
	\(\pi_{p,q}\) is in the image of \(\Gamma\). Thus, \(\Gamma\) induces an isomorphism of graph pseudogroups and thus the gluing diagram we constructed is shift surjective.
\end{proof}    

To see that whenever \(\gcd(a,n-1)=\gcd(b,n-1)\) there exists a \(l,k\) as in the above lemma, we first note that due to Bezout's theorem\cite{bachet1874problemes} there exists \(l_a,k_a\) s.t. 
	\[l_a\cdot a= k_a(n-1)+\gcd(a,n-1)=k_a(n-1)+\gcd(b,n-1).\]
Additionally, if we let \(l_b\) be the unique number s.t. \(b=l_b\gcd(b,n-1)\) we have 
	\[(l_b\cdot l_a)\cdot a=(l_b\cdot k_a)\cdot(n-1)+l_b\cdot\gcd(b,n-1)=(l_b\cdot k_a)\cdot(n-1)+b.\]
So taking \(l=l_a\cdot l_b\) and \(k=k_a\cdot l_b\) shows that \(a,b\) satisfy the conditions of the above lemma. As we can choose \(l_a\) to  be coprime to \(n-1\) and \(l_b\) is coprime to \(n-1\) by definition, we can always choose \(l\) to be coprime to \(n-1\).

All that remains is finding floating gluing diagrams connecting \(G_n\) to itself that are shift-surjective with \(x_v=l\cdot v\) for any \(l\) that is coprime to \(n-1\). We will call a pair \((l,n)\) s.t. there exists such a floating gluing diagram, \textbf{reachable}. Using expansions of gluing diagram, \defref{exp-defn}, we can show that if \((l,n)\) is reachable so is \((l+(n-1),n)\).

\begin{lemma}\label{easy-impl-lem}
	For any \(l\geq 1,n\geq 2\) if \((l,n)\) is reachable so is \((l+(n-1),n)\).
\end{lemma}
\begin{proof}
	Take a floating gluing diagram \(\mathfrak{G}=(B_v,x_v,C_e,\gamma_e)_{e\in EG_n}\) with \(x_v= l\cdot v\) and take some \(\mathbf{u}\in x_v\) that is unblocked (since \(n> 1\), such an element exists). If we look at the expanded gluing diagram \(\mathfrak{G}^{\mathbf{u}}=(y_v,C^{\mathbf{u}}_e,\gamma^{\mathbf{u}}_e)_{e\in EG_n}\), and observe that \(y_v=x^{\mathbf{u}}_v=((l-1)\cdot v+n\cdot v)=(l+(n-1))\cdot v\).
	
	If we now assume that \(\mathfrak{G}\) is shift surjective, we can take the basis \(B_R:=\{d_0,d_1,\dots, d_{l-1}\}\subseteq \mathcal{P}(G_{l,n},R)\) and endow it with a bijection \(\gamma_R:B_R\to l\cdot v\) and set \(C_{\varepsilon_R}=\{\varepsilon_R\}\). This together with \(\mathfrak{G}\) gives us a gluing diagram \(\bar{\mathfrak{G}}\) from \(G_{1,n}\) to \(G_{l,n}\) that describes a shift-preserving and shift-surjective isomorphism \(\Gamma:\mathcal{M}_p(G_{1,n},R)\to\mathcal{M}_p(G_{l,n},R)\). Now since \(\bar{\mathfrak{G}}^{\mathbf{u}}\) describes the same function and because \(G_{l,n}\) is faithful, it must also be shift-surjective by \thmref{shift-surj-thm}. And since \(\bar{\mathfrak{G}}^{\mathbf{u}}\) contains \(\mathfrak{G}^{\mathbf{u}}\), \(\mathfrak{G}^{\mathbf{u}}\) is also shift-surjective.   
	
	This shows that \((l+(n-1), n)\) is also reachable if \((l,n)\) is.
\end{proof}
Now we will also need the ability to change \(n\), for this we will introduce a way of constructing a partitioned basis of \(\mathcal{P}(G_m, x)\), from a partitioned basis of \(\mathcal{P}(G_n,x)\) for the right \(m\) and \(x\in\mathcal{FM}_{G_n}\).

\begin{defn}
	For any \(l\in\mathbb{Z}_{\geq 1}\), \(x\in \mathcal{FM}_{G_n}\subseteq\mathcal{FM}_{G_{n+l}}\), any basis \(B\subseteq \mathcal{P}(G_n,l\cdot v)\), that has a partition
		\[B=\bigsqcup_{i=0}^{n-1} C_i\]
	s.t. each \(C_i\) has a bijection \(\gamma_i:C_i\to l\cdot v\), we can define for any \(B\) internal \(p\in \mathcal{P}(G_n,l\cdot v)\)
		\[C^{p+}=\{pe_j\mid n\leq j< n+l\}\subseteq\mathcal{P}(G_{n+l},l\cdot v)\]
	this allows us to define	
		\[B^{+}=B\cup\bigsqcup_{p\not\preceq B} C^{p+}\subseteq\mathcal{P}(G_{n+l},l\cdot v).\]
\end{defn}
Note that a basis can only have finitely many internal paths, making \(B^+\) finite.
This definition will eventually allow us to construct a floating gluing diagram from \(G_{n+l}\) to itself with \(x_v=l\cdot v\), from a floating gluing diagram from \(G_{n}\) to itself with \(x_v=l\cdot v\). We will need some more lemmas first.
\begin{lemma}\label{B+-is-basis}
	For any basis \(B\subseteq \mathcal{P}(G_n,l\cdot v)\), \(B^+\subseteq \mathcal{P}(G_{n+l},l\cdot v)\) is also a basis.
\end{lemma}
\begin{proof}
	It is clear that \(B^+\) is independent, so we will just have to show that for every \(p\in \mathcal{P}(G_{n+l},l\cdot v)\) there exists some \(\tilde{p}\) with \(p\preceq \tilde{p}\) and \(B^+\preceq \tilde{p}\). For this we will take the longest \(q\in \mathcal{P}(G_{n},l\cdot v)\subseteq \mathcal{P}(G_{n+l},l\cdot v)\) s.t. \(q\preceq p\). If \(q=p\) we have \(p\in \mathcal{P}(G_{n},l\cdot v)\) giving us a \(\tilde{p}\) s.t. \(p\preceq \tilde{p}\) and \(B^+\preceq B\preceq \tilde{p}\), since \(B\) is a basis. Now assuming \(q\neq p\), we have some \(n\leq j_0<n+l\) s.t. \(qe_{j_0}\preceq p\). If \(B\preceq q\) we also have \(B^+\preceq B\preceq p\), otherwise \(q\) is internal in \(B\) and thus \(qe_{j_0}\in B^+\), showing that \(B^+\preceq p\). 
\end{proof}

To see how big the partition that we have endowed \(B^+\) with is we will have to investigate how many internal paths a basis of \(\mathcal{P}(G_{n},l\cdot v)\) has.

\begin{lemma}
	Take a basis \(B\subseteq \mathcal{P}(G_n,l\cdot v)\) then we have some \(k_B\geq 0\) s.t.
		\(|B|=(n-1)\cdot k_B+l,\)
	and \(B\) has \(k_B\) internal paths.
\end{lemma}
\begin{proof}
	We will show this inductively by first showing it for the trivial basis and if it is true for some basis \(B\) it is also true for \(B^p\), for each \(p\in B\). 

	In the case of the trivial basis \(B_0:=\{\varepsilon_{\mathbf{u}}\mid \mathbf{u}\in l\cdot v\}\) we note that \(k_{B_0}=0\) as \(|B_0|=l\) and \(B_0\) has no internal paths as each path has a prefix in \(B_0\). Now assume the lemma holds for some basis \(B\subseteq \mathcal{P}(G_n,l\cdot v)\), then for any \(p\in B\) we have
		\[|B^p|=|B|+(n-1)=(k_B+1)\cdot(n-1)+l,\]
	showing that \(k_{B^p}=k_B+1\). Additionally, we note that if \(q\) is internal in \(B^p\) it is either internal in \(B\) or equal to \(p\) (which is not internal in \(B\)), so there are \(k_B+1=k_{B^p}\) internal paths in \(B^p\).
\end{proof}

So if we have a floating gluing diagram \(\mathfrak{G}=(l\cdot v,C_e,\gamma_e)_{e\in EG_n}\) then since
	\[|B_v|=\sum_{i=0}^{n-1}|C_{e_i}|=n\cdot l=l\cdot(n-1)+1,\]
by \lemref{B+-is-basis}, \(B_v\) has \(l\) internal paths. We can thus take a bijection
	\[\rho:\{e_n,\dots,e_{n+l-1}\}\to\{p\in \mathcal{P}(G_n,l\cdot v)\mid p\text{ internal in } B_v\}\]
and a bijection \(\gamma_+: \{e_n,\dots,e_{n+l-1}\}\to l\cdot v\) and define \(\mathfrak{G}^+:=(l\cdot v,C_e,\gamma_e)_{e\in EG_{n+l}}\) by setting	
	\begin{itemize}
		\item \(C_{e_j}=C^{\rho(e_j)+}\) for each \(n\leq j < n+l\)
		\item  \(\gamma_{e_j}(\rho(e_j)e_{j'})=\gamma_+(e_{j'})\) for any \(n\leq j,j' < n+l\). 
	\end{itemize}
We  note that we have 
	\[B^{\mathfrak{G}^+}_v=(B^{\mathfrak{G}}_v)^{+}\]
for each \(v\in VG\). 
Since \(\mathfrak{G}^+\) connects \(G_{n+l}\) to itself and has \(x_v=l\cdot v\) all that remains to show is that if \(\mathfrak{G}\) is shift-surjective, then so is \(\mathfrak{G}^+\). For this we first note the following result.

\begin{lemma}\label{+bases-lem}
	For any floating gluing diagram \(\mathfrak{G}=(l\cdot v,C_e,\gamma_e)_{e\in EG_n}\) connecting \(G_n\) to itself, and any basis \(B\subseteq \mathcal{P}(G_n,v)\) we have a bijection 
		\[\rho_B:B^+\setminus B\to \{p\in\mathcal{P}(G_n,l\cdot v)\mid p \text{ internal in } \bigsqcup_{b\in B} C_b\},\]
	s.t. in the gluing diagram \(\mathfrak{G}^+\)
		\[C_{q}=\{\rho_B(q)e_j\mid n\leq j< n+l\}\]
	for any \(q\in B^+\setminus B\) and 
		\[\gamma_q(\rho_B(q)e_j)=\gamma_+(e_j)\]
	for any \(n\leq j< n+l\).

\end{lemma}
\begin{proof}
	We will proceed via induction. Firstly we note that for \(B=\{\varepsilon_v\}\), we have \(B=B^+\) and since \(C_{\varepsilon_v}=\{\varepsilon_{\mathbf{v}}\mid \mathbf{v}\in l\cdot v\}\subseteq\mathcal{P}(G_{n+l},l\cdot v)\) has no internal paths the empty function gives us the required bijection.
	For the induction step we assume we have \(\rho_B\) for some base \(B\) and use it to define \(\rho_{B^p}\). For this we note that since the paths that are internal to \(B^p\) are exactly the paths internal to \(B\), or equal to \(p\), we have
		\[(B^p)^+\setminus B^p=B^+\setminus B\cup \{pe_j\mid n\leq j< n+l\}. \]
	We can thus define \(\rho_{B^p}\) by setting \(\rho_{B^p}|_{B^+\setminus B}=\rho_B\) and 
		\[\rho_{B^p}(pe_j)=\gamma^{-1}_p(\rho(e_j))\]
	for any \(n\leq j <n+l\) where we view \(\gamma_p\) as a function from \(\mathcal{P}(G_{n+l},v)_{C_p}\) to \(\mathcal{P}(G_{n+l},l\cdot v)\). Since \(\gamma_p\) is a prefix respecting bijection, we note that since \(B_v\not\preceq\rho(e_j)\) we have \(C_p\circ_{\gamma_p} B_{T(p)}\not\preceq \gamma^{-1}_p(\rho(e_j))\). Additionally, since \(C_p\preceq \gamma^{-1}_p(\rho(e_j))\), we must also have \(C_q\not\preceq \gamma^{-1}_p(\rho(e_j))\) for any \(p\neq q\in B\) and thus
		\[\bigcup_{q\in B^p}C_q=\bigcup_{q\in B\setminus \{p\}} C_q\cup (C_p\circ_{\gamma_p} B_{T(p)})\not\preceq \gamma^{-1}_p(\rho(e_j))\]
	i.e. \(\rho_{B^p}(pe_j)\) is internal in \(\bigcup_{q\in B^p}C_q\). As any path \(r\) that is internal in  \(\bigcup_{q\in B^p}C_q\) is either internal in \(\bigcup_{q\in B}C_q\) or satisfies 
	\begin{multline*}
		C_p\preceq r \text{ and }C_p\circ_{\gamma_p} B_v\not\preceq r\iff B_v\not\preceq \gamma_p(r)\iff\exists n\leq j<n+j,\ \gamma_p(r)=\rho(e_j)\\ \iff\exists n\leq j<n+j,\ r=\gamma_p^{-1}(\rho(e_j))= \rho_{B^p}(pe_j)
	\end{multline*}
		 
	and \(\rho_B,\rho\) are bijections, we can see that \(\rho_{B^p}\) is also a bijection. 

	Now we note that for any \(q\in (B^p)^+\setminus B^p\) with \(p\not\preceq q\) we have \(q\in B^+\setminus B\) and thus by induction 
		\[C_q=\{\rho_B(q)e_j\mid n\leq j< n+l\}=\{\rho_{B^p}(q)e_j\mid n\leq j< n+l\}\]
	and \(\gamma_q(\rho_{B^p}(q)e_j)=\gamma_q(\rho_B(q)e_j)=\gamma_+(e_j)\).
	On the other hand any \(q\in (B^p)^+\setminus B^p\) with \(p\preceq q\) is of the form \(pe_j\) for some \(n\leq j < n+l\) and thus
	\begin{multline*}
		C_q=C_{pe_j}=C_p\circ_{\gamma_p} C_{e_j}=\gamma^{-1}_p(\{\rho(e_j)e_k\mid n\leq k < n+l \})=\\
		\{\gamma^{-1}_p(\rho(e_j))e_k\mid n\leq k < n+l \}=\{\rho_{B^p}(pe_j)e_k\mid n\leq k < n+l \}
	\end{multline*}
	and for each \(n\leq k < n+l\)
		\[\gamma_{pe_j}(\rho_{B^p}(pe_j)e_k)=\gamma_{pe_j}(\gamma_p^{-1}(\rho(e_j)e_k))=\gamma_{e_j}(\rho(e_j)e_k)=\gamma_+(e_k).\]

\end{proof}

This gives us the following corollary
 
\begin{cor}\label{+bases-cor}
	For any iterated basis \(B\subseteq \mathcal{P}(G_n,l\cdot v)\) of some gluing diagram \(\mathfrak{G}\), \(B^+\) is an iterated basis of \(\mathfrak{G}^+\). 
\end{cor}
\begin{proof}
	Take a basis \(B'\subseteq \mathcal{P}(G,v)\) s.t. \( B=\bigcup_{p\in B'} C_p\), we then have using \lemref{+bases-lem}
		\begin{multline*}
		\bigcup_{p\in (B')^+}C_p=B\cup \bigcup_{p\in (B')^+\setminus B'} C_p=B\cup\{\rho_B(p)e_j\mid p\in (B')^+\setminus B',\ n\leq j< n+l\}\\
		=B\cup \{qe_j\mid q\text{ internal in } B,\ n\leq j< n+l\}=B^+,
		\end{multline*}
	making \(B^+\) an iterated basis.
\end{proof}

\begin{lemma}\label{surj+-lem}
	For any surjective gluing diagram \(\mathfrak{G}\) from \(G_{n}\) to itself with \(x_v=l\cdot v\), the floating gluing diagram \(\mathfrak{G}^+\) is also surjective.
\end{lemma}
\begin{proof}
	To show that \(\mathfrak{G}^+\) is unblocked, we will show that for any \(\mathbf{u},\mathbf{w}\in l\cdot v\) s.t. \(\mathbf{w}\) blocks \(\mathbf{u}\) in \(\mathfrak{G}^{+}\), \(\mathbf{w}\) blocks \(\mathbf{u}\) in \(\mathfrak{G}\) as well. To see this observe that for any \(n\leq j < n+l\) we have \(C_{e_j}= \{\rho(e_j)e_k\mid n\leq k<n+l\} \) and thus \(\varepsilon_{\mathbf{u}}\notin C_{e_j}\). Thus, we must have some \(0\leq k <n\) s.t. \(\varepsilon_{\mathbf{u}}\in C_{e_k}\) and since \(\gamma_{e_k}(\varepsilon_{\mathbf{u}})=\mathbf{w}\) we see that \(\mathbf{w}\) blocks \(\mathbf{u}\) in \(\mathfrak{G}\).
	
	To construct for any \(\mathbf{u}\in l\cdot v\), a \(\mathbf{u}\)-splitting iterated basis, we first take a basis \(B'\subseteq\mathcal{P}(G_n,v)\) s.t. \(B:=\bigcup_{p\in B'} C_p\) is \(\mathbf{u}\)-splitting. We will show that \(B^+\) is also splitting. For this we note that if we take \(D_{\mathbf{u}}\subseteq B'\) s.t. \(\bigcup_{p\in D_{\mathbf{u}}}C_p\) covers \(\varepsilon_{\mathbf{u}}\) exactly in \(\mathcal{P}(G_n,l\cdot v)\). We then note that for \(D^{+}_{\mathbf{u}}:=D_{\mathbf{u}}\cup\{p\in (B')^+\setminus B'\mid \varepsilon_{\mathbf{u}}\preceq\rho_{B'}(p)\}\) we have 
		\[\bigcup_{p\in D^+_{\mathbf{u}}}C_p=\bigcup_{p\in D_{\mathbf{u}}}C_p\cup \{pe_j\mid p\text{ is internal in } B, \varepsilon_{\mathbf{u}}\preceq p\}.\]
	This set covers \(\varepsilon_{\mathbf{u}}\) exactly since it clearly lies under it and 
		\[\{p\in B^+\mid \varepsilon_{\mathbf{u}}\preceq p\}=\bigcup_{p\in D_{\mathbf{u}}}C_p\cup \{pe_j\mid p\text{ is internal in } B, \varepsilon_{\mathbf{u}}\preceq p\}.\]

\end{proof}

\begin{lemma}\label{shift-surj+-lem}
	For any shift-surjective floating gluing diagram \(\mathfrak{G}\) from \(G_{n}\) to itself with \(x_v=l\cdot v\), the floating gluing diagram \(\mathfrak{G}^+\) is also shift-surjective.
\end{lemma}
\begin{proof}

	Take some \(p,q\) internal in \(B^+_v\) and note that they must also be internal in \(B_v\). Since \(\mathfrak{G}\) is shift surjective, we have independent \(D_p,D_q\subseteq\mathcal{P}(G,v)\) and terminus-maintaining bijection \(\nu_{p,q}:D_p\to D_q\) that enable the shift from \(p\) to \(q\). Take bases \(B_p,B_q\subseteq\mathcal{P}(G_n,v)\) with \(D_p\subseteq B_p,D_q\subseteq B_q\) we will define 
		\begin{align*}
			D^+_p:&=D_p\cup \{\tilde{p}\in B^+_p\setminus B_p\mid p\preceq\rho_{B_p}(\tilde{p})\}\\
			D^+_q:&=D_q\cup \{\tilde{q}\in B^+_q\setminus B_q\mid q\preceq\rho_{Bq}(\tilde{q})\}\\
			\nu^+_{p,q}:&\begin{cases}
							D^+_p\to D^+_q\\
							\tilde{p}\mapsto\begin{cases}
												\nu_{p,q}(\tilde{p}),& \tilde{p}\in D_p\\
												\rho_{B_q}^{-1}(qr),& \tilde{p}\in D^+_p\setminus D_p, r\in \mathcal{P}(G_{n+l},v), \text{ s.t. } \rho_{B_p}(\tilde{p})=pr
											\end{cases}
						\end{cases}
		\end{align*}
	and show that they enable the shift from \(p\) to \(q\) in \(\mathfrak{G}^+\). In the same way as in \lemref{surj+-lem} we can see that \(\bigcup_{\tilde{p}\in D^{+}_p} C_{\tilde{p}}\) and \(\bigcup_{\tilde{q}\in D^{+}_q} C_{\tilde{q}}\) exactly cover \(p\) and \(q\) respectively. Now if we have some \(\tilde{p}\in D^+_p\) and some \(r\in\mathcal{P}(G_{n+l},v)\) s.t \(pr\in C_{\tilde{p}}\). If \(\tilde{p}\in D_p\), then \(r\in\mathcal{P}(G_{n},v)\) and since \(D_p,D_q,\nu_{p,q}\) enable the shift from \(p\) to \(q\), we have \(qr\in C_{\nu_{p,q}(\tilde{p})}=C_{\nu^+_{p,q}(\tilde{p})}\) and additionally
		\[\gamma_{\tilde{p}}(pr)=\gamma_{\nu_{p,q}(\tilde{p})}(qr)=\gamma_{\nu^+_{p,q}(\tilde{p})}(qr).\]
	In the opposite case where \(\tilde{p}\in D^{+}_p\setminus D_p\) then if we write \(\rho_{B_p}(\tilde{p})=pr^-\) for some \(r^-\in\mathcal{P}(G_n,v)\). Since we have \(pr\in C_{\tilde{p}}\), we must have some \(n\leq j<n+l\) s.t. \(r=r^-e_j\) and since \(qr=qr^-e_j\in C_{\rho_{B_q}^{-1}(qr^{-})}\), we must have \(qr\in C_{\nu_{p,q}(\tilde{p})}\) by definition. Additionally, we have
		\[\gamma_{\tilde{p}}(pr)=\gamma_{\tilde{p}}(\rho_{B_p}(\rho_{B_p}^{-1}(pr^{-}))e_j)=\gamma_+(e_j)=\gamma_{\nu_{p,q}(\tilde{p})}(\rho_{B_q}(\rho_{B_q}^{-1}(qr^{-}))e_j)=\gamma_{\nu_{p,q}(\tilde{p})}(qr).\]
	This shows that \(\mathfrak{G}^+\) enables the shift from \(p\) to \(q\) and since \(p,q\) were arbitrary paths internal in \(B^+_v\), \(\mathfrak{G}^+\) is shift-surjective.

\end{proof}

We can now combine the above lemmas to get the following.

\begin{lemma}\label{hard-impl-lem}
	If \((l,n)\) reachable so is \((l,n+l)\).
\end{lemma}
\begin{proof}
	This lemma follows after combining \lemref{surj+-lem} and \lemref{shift-surj+-lem}.
\end{proof}

Finally, we can prove the following.
\begin{lemma}\label{reach-lem}
	For any \(l\geq 1\) and \(n\geq 2\), s.t. \(\gcd(l,n-1)=1\), \((l,n)\) is reachable.
\end{lemma}
\begin{proof}
	This proof we will be proceeding as in the Euclidean Algorithm. To this end we will use induction over \(n+l\). For the induction beginning we note that \((1,2)\) is reachable since the floating gluing diagram \(\mathfrak{G}=(B_v,v,C_e,\gamma_e)_{e\in EG_2}\) with 
		\[
			B_v:=\{e_1,e_2\},\quad
			C_{e_1}=\{e_1\},\quad, C_{e_2}=\{e_2\}\]
	and \(\gamma_{e_1},\gamma_{e_2}\) the only possible functions, is clearly shift surjective.

	For the induction step we distinguish \(3\) cases: \(l< n-1\),  \(n-1=l\) and \(n-1< l\).

	If \(l< n-1\), then \(n-l\geq 2\) and \(\gcd(l,n-l-1)=1\), so by the induction hypothesis \((l,n-l)\) is reachable and so by \lemref{hard-impl-lem}, \((l,n)\) is reachable.

	If \(n-1= l\), then \(n-1=l=\gcd(l,n-1)=1\), and thus we have \((l,n)=(1,2)\), the case already covered by the induction beginning.
	
	If \(n-1< l\) we note that \(l-(n-1)\geq 1\) and \(\gcd(l-(n-1),n-1)=1\) so by induction \((l-(n-1), n)\) is reachable. So by \lemref{easy-impl-lem} \((l,n)\) is reachable. 
\end{proof}

This directly gives us the main result of the section.
\begin{theorem}
	For any \(a,b\in\mathbb{Z}_{\geq 1}\) and \(n,m\in\mathbb{Z}_{\geq 2}\) if there is a shif-preserving and shift-surjective isomorphism
		\[\Gamma:\mathcal{M}_p(G_{a,n},R)\to \mathcal{M}_p(G_{a,n},R),\]
	if and only if
		\[m=n \text{ and }\gcd(a,n-1)=\gcd(b,m-1).\]
\end{theorem}
\begin{proof}
	The "only if" implication was proven in \lemref{first-impl-lem}. The second implication follows from combining \lemref{reach-lem} and \lemref{floating-to-rooted-lem}.
\end{proof}

The above proof therefore solves the isomorphism problem of Higman-Thompson groups and by examining the proofs of the above lemmas, especially \lemref{reach-lem}, we in fact get a way of constructing gluing diagramms between \(G_{a,n}\) and \(G_{b,n}\) whenever \(\gcd(a,n-1)=\gcd(b,n-1)\). This procedure first constructs \(B_R\) as in \lemref{floating-to-rooted-lem} and to constructs \(B_v\), as in \lemref{reach-lem}.  

As an example of this procedure we will construct a shift-surjective gluing diagram connecting \(G_{3,5}\) and \(G_{1,5}\). For this we will note that 
	\[3\cdot 3=(5-1)\cdot 2 + 1\]
This means that we have a basis in \(\mathcal{P}(G_{1,5},R)\) that can be partitioned into \(3\) sets of cardinality \(3\), this basis will be our \(B_R\) with the partition \(C_{d_1},C_{d_2},C_{d_3},C_{d_4}\) being arbitrary, and arbitrary functions from \(C_{d_j}\) to \(l\cdot v\) being \(\gamma_{d_j}\) for each \(1\leq j\leq 4\). Visually we can represent it as follows.

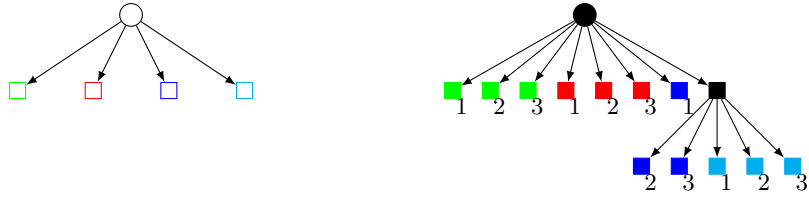
\begin{center}
	\begin{tikzpicture}
		[edge from parent/.style={draw,-latex},
			level distance=10mm,
			level 1/.style={sibling distance=5mm},
			level 2/.style={sibling distance=7mm,},
			level 3/.style={sibling distance=5mm,},
			label distance=-2mm]
		\node[big hollow circle] (r1) at (-2,0) {}
			child[sibling distance=10mm] {node[big hollow green square] {}}
			child[sibling distance=10mm] {node[big hollow red square] {}}
			child[sibling distance=10mm] {node[big hollow blue square] {}}
			;

		\node[big solid circle] (r21) at (4,0)	{}
			child {node[big solid square] {}
					child {node[big solid green square, label={[font=\small,text=black]-10:$1$}] {}}
					child {node[big solid green square, label={[font=\small,text=black]-10:$2$}] {}}
					child {node[big solid green square, label={[font=\small,text=black]-10:$3$}] {}}
					child {node[big solid red square, label={[font=\small,text=black]-10:$1$}] {}}
					child {node[big solid square] {}
						child {node[big solid red square, label={[font=\small,text=black]-10:$2$}] {}}
						child {node[big solid red square, label={[font=\small,text=black]-10:$3$}] {}}
						child {node[big solid blue square, label={[font=\small,text=black]-10:$1$}] {}}
						child {node[big solid blue square, label={[font=\small,text=black]-10:$2$}] {}}
						child {node[big solid blue square, label={[font=\small,text=black]-10:$3$}] {}}}
				}
			;
		
	\end{tikzpicture}
\end{center}

To construct \(B_v\) we will need to construct a floating gluing diagram from \(G_5\) to itself with \(x_v=3\cdot v\). We will do this in \lemref{reach-lem}, by looking at the series of \((l,n)\)'s
	\[(1,2)\to(2,2)\to(3,2)\to(3,5)\]
where at each arrow represent either going from \((l,n)\) to \((l+(n-1),n)\), or going from \((l,n)\) to \((l,n+l)\), which can be accomplished by expansion or addition respectively, of the previously constructed gluing diagram. We will start with \(\mathfrak{G}_0\) being the trivial gluing diagram connecting \(G_2\) to itself with \(x_v=v\), which can be represented visually as	follows.
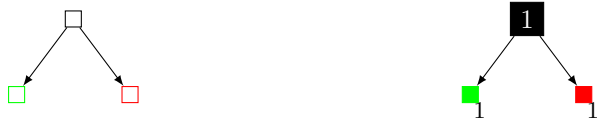
\begin{center}
	\begin{tikzpicture}
		[edge from parent/.style={draw,-latex},
			level distance=10mm,
			level 1/.style={sibling distance=15mm},
			level 2/.style={sibling distance=5mm,},
			level 3/.style={sibling distance=3mm,},
			label distance=-2mm]
		\node[big hollow square] (r1) at (-2,0) {}
			child[sibling distance=15mm] {node[big hollow green square] {}}
			child[sibling distance=15mm] {node[big hollow red square] {}};

		\node[big solid square] (r21) at (4,0)	{\textcolor{white}{$1$}}
			child {node[big solid green square, label={[font=\small,text=black]-10:$1$}] {}}
			child {node[big solid red square, label={[font=\small,text=black]-10:$1$}] {}};		
	\end{tikzpicture}
\end{center}
To go from \((1,2)\) to \((2,2)\), we get a gluing set \(\mathfrak{G}_1\) from \(G_2\) to itself by expanding over the unique element \(\mathbf{v}\in x_v=v\), i.e. \(\mathfrak{G}_1=\mathfrak{G}^{\mathbf{v}}_0\), which can be represented visually as follows.

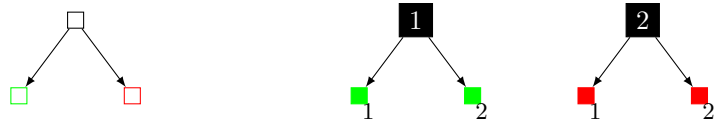
\begin{center}
	\begin{tikzpicture}
		[edge from parent/.style={draw,-latex},
			level distance=10mm,
			level 1/.style={sibling distance=15mm},
			level 2/.style={sibling distance=5mm,},
			level 3/.style={sibling distance=3mm,},
			label distance=-2mm]
		\node[big hollow square] (r1) at (-2,0) {}
			child[sibling distance=15mm] {node[big hollow green square] {}}
			child[sibling distance=15mm] {node[big hollow red square] {}};

		\node[big solid square] (r21) at (2.5,0)	{\textcolor{white}{$1$}}
			child {node[big solid green square, label={[font=\small,text=black]-10:$1$}] {}}
			child {node[big solid green square, label={[font=\small,text=black]-10:$2$}] {}};	
		\node[big solid square] (r21) at (5.5,0)	{\textcolor{white}{$2$}}
			child {node[big solid red square, label={[font=\small,text=black]-10:$1$}] {}}
			child {node[big solid red square, label={[font=\small,text=black]-10:$2$}] {}};	
	\end{tikzpicture}
\end{center}
This way we get \(\mathfrak{G}_1\), a gluing set from \(G_2\) to  itself with \(x_v=2\cdot v\). To go from \((2,2)\) to \((3,2)\) we will perform another expansion over some \(\mathbf{v}\in x_v=2\cdot v\) (we will take the square with number \(1\), but they both are unblocked) and get \(\mathfrak{G}_2:=\mathfrak{G}^{\mathbf{v}}_1\), which we can represent visually as follows.

\begin{center}
	\begin{tikzpicture}
		[edge from parent/.style={draw,-latex},
			level distance=10mm,
			level 1/.style={sibling distance=10mm},
			level 2/.style={sibling distance=5mm,},
			level 3/.style={sibling distance=3mm,},
			label distance=-2mm]
		\node[big hollow square] (r1) at (-1,0) {}
			child[sibling distance=15mm] {node[big hollow green square] {}}
			child[sibling distance=15mm] {node[big hollow red square] {}}
			;

		\node[big solid square] (r21) at (3.5,0)	{\textcolor{white}{$1$}}
			child {node[big solid green square, label={[font=\small,text=black]-10:$1$}] {}}
			child {node[big solid green square, label={[font=\small,text=black]-10:$2$}] {}}
			;
		\node[big solid green square, label={[font=\small,text=black]-10:$3$}] (r21) at (5.25,0)	{\textcolor{white}{$2$}};	
		\node[big solid square] (r21) at (7,0)	{\textcolor{white}{$3$}}
			child {node[big solid square] {}
				child[sibling distance=10mm] {node[big solid red square, label={[font=\small,text=black]-10:$1$}] {}}
				child[sibling distance=10mm] {node[big solid red square, label={[font=\small,text=black]-10:$2$}] {}}
				}
			child {node[big solid red square, label={[font=\small,text=black]-10:$3$}] {}}
			;	
	\end{tikzpicture}
\end{center}
So we get a gluing diagram \(\mathfrak{G}_2\) from \(G_2\) to itself with \(x_v=3\cdot v\). To go from \((3,2)\) to \((3,5)\) we will perform addition on \(\mathfrak{G}_2\) and get \(\mathfrak{G}_3:=\mathfrak{G}^+_2\). This we can represent visually as follows.
	
\begin{center}
	\begin{tikzpicture}
		[edge from parent/.style={draw,-latex},
			level distance=10mm,
			level 1/.style={sibling distance=5mm},
			level 2/.style={sibling distance=5mm,},
			level 3/.style={sibling distance=3mm,},
			label distance=-2mm]
		\node[big hollow square] (r1) at (-1,0) {}
			child[sibling distance= 10mm] {node[big hollow green square] {}}
			child[sibling distance= 10mm] {node[big hollow red square] {}}
			child[sibling distance= 10mm] {node[big hollow blue square] {}}
			child[sibling distance= 10mm] {node[big hollow yellow square] {}}
			child[sibling distance= 10mm] {node[big hollow cyan square] {}}
			;

		\node[big solid square] (r21) at (3.5,0)	{\textcolor{white}{$1$}}
			child {node[big solid green square, label={[font=\small,text=black]-10:$1$}] {}}
			child {node[big solid green square, label={[font=\small,text=black]-10:$2$}] {}}
			child {node[big solid blue square, label={[font=\small,text=black]-10:$1$}] {}}
			child {node[big solid blue square, label={[font=\small,text=black]-10:$2$}] {}}
			child {node[big solid blue square, label={[font=\small,text=black]-10:$3$}] {}}
			;
		\node[big solid green square, label={[font=\small,text=black]-10:$3$}] (r21) at (5.25,0)	{\textcolor{white}{$2$}};	
		\node[big solid square] (r21) at (7,0)	{\textcolor{white}{$3$}}
			child {node[big solid square] {}
				child {node[big solid red square, label={[font=\small,text=black]-10:$1$}] {}}
				child {node[big solid red square, label={[font=\small,text=black]-10:$2$}] {}}
				child {node[big solid yellow square, label={[font=\small,text=black]-10:$1$}] {}}
				child {node[big solid yellow square, label={[font=\small,text=black]-10:$2$}] {}}
				child {node[big solid yellow square, label={[font=\small,text=black]-10:$3$}] {}}}
			child {node[big solid red square, label={[font=\small,text=black]-10:$3$}] {}}
			child {node[big solid cyan square, label={[font=\small,text=black]-10:$1$}] {}}
			child {node[big solid cyan square, label={[font=\small,text=black]-10:$2$}] {}}
			child {node[big solid cyan square, label={[font=\small,text=black]-10:$3$}] {}}
			;	
	\end{tikzpicture}
\end{center}

So if we combine this floating gluing diagram \(\mathfrak{G}_3\) with the partitioned basis \(B_R\) of \(\mathcal{P}(G_{1,5},R)\) and set the starting basis to be \(\{\varepsilon_R\}\), we get a gluing diagram \(\mathfrak{G}\) connecting \((G_{3,5},R)\) and \((G_{1,5},R)\). Since in the construction of \(\mathfrak{G}_{3}\) started with the shift surjective identity diagram and applied expansions and the \(+\) operator, which preserve shift surjectivity, the diagrams \(\mathfrak{G}_3\) and \(\mathfrak{G}\) are both shift surjective. Thus, \(\mathfrak{G}\) gives us an explicit isomorphism between the Higman-Thompson groups \(V_{3,5}\) and \(V_{1,5}\).
\begin{figure}
	\begin{center}
		\begin{tikzpicture}
		[edge from parent/.style={draw,-latex},
			level distance=10mm,
			level 1/.style={sibling distance=5mm},
			level 2/.style={sibling distance=7mm,},
			level 3/.style={sibling distance=5mm,},
			label distance=-2mm]
		\node[big hollow circle] (r1) at (-2,0) {}
			child[sibling distance=10mm] {node[big hollow green square] {}}
			child[sibling distance=10mm] {node[big hollow red square] {}}
			child[sibling distance=10mm] {node[big hollow blue square] {}}
			;

		\node[big solid circle] (r21) at (4,0)	{}
			child {node[big solid square] {}
					child {node[big solid green square, label={[font=\small,text=black]-10:$1$}] {}}
					child {node[big solid green square, label={[font=\small,text=black]-10:$2$}] {}}
					child {node[big solid green square, label={[font=\small,text=black]-10:$3$}] {}}
					child {node[big solid red square, label={[font=\small,text=black]-10:$1$}] {}}
					child {node[big solid square] {}
						child {node[big solid red square, label={[font=\small,text=black]-10:$2$}] {}}
						child {node[big solid red square, label={[font=\small,text=black]-10:$3$}] {}}
						child {node[big solid blue square, label={[font=\small,text=black]-10:$1$}] {}}
						child {node[big solid blue square, label={[font=\small,text=black]-10:$2$}] {}}
						child {node[big solid blue square, label={[font=\small,text=black]-10:$3$}] {}}}
				}
			;
		
	\end{tikzpicture}
\vspace{1cm}
	
	\begin{tikzpicture}
		[edge from parent/.style={draw,-latex},
			level distance=10mm,
			level 1/.style={sibling distance=5mm},
			level 2/.style={sibling distance=5mm,},
			level 3/.style={sibling distance=3mm,},
			label distance=-2mm]
		\node[big hollow square] (r1) at (-1,0) {}
			child[sibling distance= 10mm] {node[big hollow green square] {}}
			child[sibling distance= 10mm] {node[big hollow red square] {}}
			child[sibling distance= 10mm] {node[big hollow blue square] {}}
			child[sibling distance= 10mm] {node[big hollow yellow square] {}}
			child[sibling distance= 10mm] {node[big hollow cyan square] {}}
			;

		\node[big solid square] (r21) at (3.5,0)	{\textcolor{white}{$1$}}
			child {node[big solid green square, label={[font=\small,text=black]-10:$1$}] {}}
			child {node[big solid green square, label={[font=\small,text=black]-10:$2$}] {}}
			child {node[big solid blue square, label={[font=\small,text=black]-10:$1$}] {}}
			child {node[big solid blue square, label={[font=\small,text=black]-10:$2$}] {}}
			child {node[big solid blue square, label={[font=\small,text=black]-10:$3$}] {}}
			;
		\node[big solid green square, label={[font=\small,text=black]-10:$3$}] (r21) at (5.25,0)	{\textcolor{white}{$2$}};	
		\node[big solid square] (r21) at (7,0)	{\textcolor{white}{$3$}}
			child {node[big solid square] {}
				child {node[big solid red square, label={[font=\small,text=black]-10:$1$}] {}}
				child {node[big solid red square, label={[font=\small,text=black]-10:$2$}] {}}
				child {node[big solid yellow square, label={[font=\small,text=black]-10:$1$}] {}}
				child {node[big solid yellow square, label={[font=\small,text=black]-10:$2$}] {}}
				child {node[big solid yellow square, label={[font=\small,text=black]-10:$3$}] {}}}
			child {node[big solid red square, label={[font=\small,text=black]-10:$3$}] {}}
			child {node[big solid cyan square, label={[font=\small,text=black]-10:$1$}] {}}
			child {node[big solid cyan square, label={[font=\small,text=black]-10:$2$}] {}}
			child {node[big solid cyan square, label={[font=\small,text=black]-10:$3$}] {}}
			;	
	\end{tikzpicture}
\end{center}
\caption{A shift surjective diagram connecting \(G_{3,5}\) and \(G_{1,5}\)}
\end{figure}
\bibliographystyle{plain}
\bibliography{bibliography}
\end{document}